\documentclass{article}

\usepackage[english]{babel}

\usepackage[letterpaper,top=2cm,bottom=2cm,left=3cm,right=3cm,marginparwidth=1.75cm]{geometry}

\usepackage{amsmath}
\usepackage{amssymb}

\IfFileExists{euscript.sty}{\usepackage[mathcal]{euscript}}{} 
\usepackage[all]{xy}
\usepackage{amsthm}
\usepackage{graphicx}
\usepackage[colorlinks=true, allcolors=blue]{hyperref}
\usepackage{enumerate}

\newcommand{\nc}{\newcommand}
\nc{\kk}{\mathbf{k}}
\newcommand{\DMO}{\DeclareMathOperator}

\nc{\eqn}{\begin{equation}}
\nc{\eqnn}{\begin{equation*}}
\nc{\eqnd}{\end{equation}}
\nc{\eqnnd}{\end{equation*}}
\nc{\enum}{\begin{enumerate}}
\nc{\enumd}{\end{enumerate}}

\DMO{\str}{str}
\DMO{\htpy}{htpy}
\DMO{\id}{id}
\DMO{\ob}{ob}
\DMO{\ho}{ho}

\nc{\eqs}{\mathbf{w}}
\nc{\quis}{\mathbf{quis}}

\nc{\Cat}{\mathsf{Cat}}
\nc{\Mod}{\mathsf{Mod}}
\nc{\inftyCat}{\mathcal{C}\!\operatorname{at}_\infty}
\nc{\winftyCat}{\mathcal{WC}\!\operatorname{at}_\infty}
\nc{\inftyGpd}{\mathcal{G}\!\operatorname{pd}_\infty}
\nc{\Ainftycat}{\mathcal{C}\!\operatorname{at}_{A_\infty}}
\nc{\dgcat}{\mathcal{C}\!\operatorname{at}_{dg}}
\nc{\Ainftycatt}{A_\infty Cat}
\nc{\dgcatt}{dg Cat}
\nc{\dg}{\mathsf{dg}}
\nc{\simpcatt}{SimpCat}
\nc{\spaces}{\mathcal{S}\!\operatorname{paces}}
\nc{\algcat}{\operatorname{Alg}_{\operatorname{cat}}}
\nc{\Set}{\operatorname{Set}}
\DMO{\End}{End}
\DMO{\bimod}{Bimod}

\nc{\homcof}{\operatorname{cof}}
\nc{\homproj}{\operatorname{hproj}}

\DMO{\FFES}{FFES}

\nc{\Ainftystr}{\Ainfty^{\str}}
\nc{\Ainfty}{{\mathbb A}_\infty}
          \DMO{\op}{op}
          \DMO{\tw}{Tw}
          \DMO{\Map}{Map}
          \DMO{\aut}{Aut}
          \DMO{\cone}{Cone}
          \DMO{\colim}{colim}
          \DMO{\stable}{Ex}
          \DMO{\hocolim}{hocolim}	
          \nc{\xto}{\xrightarrow}
          \nc{\xra}{\xto}
          \nc{\tensor}{\otimes}
          \nc{\Tw}{\mathsf{Tw}}
          \nc{\into}{\hookrightarrow}
          \nc{\sets}{\mathsf{Sets}}
          \nc{\sing}{\mathsf{Sing}}
          \nc{\sset}{\mathsf{sSet}}
          \nc{\chain}{\mathsf{Chain}}
          \nc{\fun}{\mathsf{Fun}}

\nc{\ba}{\mathbf a}
\nc{\bq}{\mathbf q}
\nc{\bx}{\mathbf x}
\nc{\bw}{\mathbf w}
\nc{\unQ}{\underline{Q}}

\nc{\Aploc}{A^+[\id_A^{-1}]}
\nc{\Aploco}{A^+[\id_0^{-1}]}

     \def\cA{\mathcal A} \def\cB{\mathcal B} \def\cC{\mathcal C} \def\cD{\mathcal D} \def\cE{\mathcal E}                 \def\cV{\mathcal V} \def\cW{\mathcal W}  \def\cY{\mathcal Y}  
     
             \def\II{\mathbb I} \def\JJ{\mathbb J}  \def\LL{\mathbb L}       \def\SS{\mathbb S} \def\TT{\mathbb T}      \def\ZZ{\mathbb Z} 
     
        \def\sD{\mathsf D}             \def\sq{\mathsf q}

 \def\kk{\mathbf k}

\nc{\hiro}{\textcolor{blue}}
\nc{\done}{\vskip1em\textcolor{red}{$\heartsuit\heartsuit$}}
          \theoremstyle{definition}
          \newtheorem{theorem}{Theorem}[section]
          
          \newtheorem{prop}[theorem]{Proposition}
          \newtheorem{lemma}[theorem]{Lemma}

          \newtheorem{corollary}[theorem]{Corollary}
          \newtheorem{construction}[theorem]{Construction}

          \newtheorem{defn}[theorem]{Definition}
          \newtheorem{notation}[theorem]{Notation}
          \newtheorem{example}[theorem]{Example}

          \newtheorem{recollection}[theorem]{Recollection}
          \newtheorem{remark}[theorem]{Remark}
          \newtheorem{convention}[theorem]{Convention}

\title{Unitalities and mapping spaces in $A_\infty$-categories}
\author{Hiro Lee Tanaka}

\begin{document}
\maketitle

\begin{abstract}
We prove, over any base ring, that the infinity-category of strictly unital A-infinity-categories (and strictly unital functors) is equivalent to the infinity-category of unital A-infinity-categories (and unital functors). We also identify various models for internal homs and mapping spaces in the infinity-categories of dg-categories and of A-infinity--categories, generalizing results of To\"en and Faonte.
\end{abstract}

\tableofcontents

\section{Introduction}

Fix a commutative unital base ring $\kk$ and let $A$ be an $A_\infty$-category over $\kk$. We recall two notions of unitality.

Given an object $X$, a closed degree-zero element $e \in \hom_A(X,X)$ is called a {\em unit} if for every object $W$, the chain maps
	\eqn\label{eqn. compositions with e}
	m^2(e,-): \hom_A(W,X) \to \hom_A(W,X) \qquad\text{and}\qquad m^2(-,e): \hom_A(X,W) \to \hom_A(X,W)
	\eqnd
are homotopic to the identity chain maps.
$A$ is called unital if every object $X$ admits a unit, and a functor between two unital $A_\infty$-categories is called unital if units are sent to units. We let
	$
	\Ainfty
	$
denote the category (in the classical sense) of unital $A_\infty$-categories and unital functors.

A closed degree-zero $e \in \hom_A(X,X)$ is a {\em strict} unit if $m^n$ ($n \geq 3$) vanishes whenever at least one of the $n$ inputs is $e$, and if the maps~\eqref{eqn. compositions with e} are the identity chain maps. We say that $A$ is {\em strictly unital} if every object admits a strict unit, and a functor among strictly unital $A_\infty$-categories is called strictly unital if the image of every strict unit is a strict unit. $\Ainftystr$ denotes the category (in the classical sense)  of strictly unital $A_\infty$-categories and strictly unital functors. 

Finally, a functor $f$ of unital $A_\infty$-categories is called a {\em quasi-equivalence} if $f$ is an equivalence of cohomology categories (i.e., essentially surjective on cohomology categories, and the maps of morphism complexes are all quasi-isomorphisms). Abusing notation, we let
	\eqnn
	\eqs \subset \Ainftystr,
	\qquad
	\eqs \subset \Ainfty
	\eqnnd
denote the collection of quasi-equivalences in each category. The $\eqs$ is meant to convey the term ``weak equivalences'' from homotopy theory, or from the theory of model categories. 

\subsection{Main results}
Our first main result states that the $\infty$-category of strictly unital $A_\infty$-categories is equivalent to the $\infty$-category of unital $A_\infty$-categories. 
We refer to Remark~\ref{recollection. localizations} for more on localizations of $\infty$-categories.

\begin{theorem}\label{theorem. main theorem}
The inclusion $j: \Ainftystr \into \Ainfty$ induces an equivalence of $\infty$-categorical localizations
	\eqnn
	\Ainftystr[\eqs^{-1}] \to \Ainfty[\eqs^{-1}].
	\eqnnd
\end{theorem}

\begin{remark}
The proof method of Theorem~\ref{theorem. main theorem} is different from that employed by Pascaleff (and later by Canonaco-Ornaghi-Stellari)~\cite{pascaleff,cos-over-rings}. The cited works utilize a Dwyer-Kan adjunction between relative categories, but in practice one need not use the full power of an adjunction. Indeed, a basic fact about localizations of $\infty$-categories (manifested as Proposition~\ref{prop. good natural transformations induce homotopies} below) allows one to construct left and right inverses to the putative equivalence of $\infty$-categories, without any demand that the left and right inverses be adjoints to the putative equivalence -- nor equal! The intuition is that if a natural transformation $\eta$, before passing to localizations, happens to consist of morphisms in $\eqs$, then $\eta$ must induce a natural {\em equivalence} after localization.
\end{remark}

Now let $\dgcatt \to \Ainfty^{\str}$ be the inclusion of dg-categories and dg-functors into strictly unital $A_\infty$-categories and strictly unital $A_\infty$ functors. Results of Pascaleff (over a field) and Canonaco-Ornaghi-Stellari (over a commutative ring)~\cite{pascaleff, cos-over-rings} show that this inclusion induces an equivalence of $\infty$-categories upon localization. Thus, an immediate corollary of 
Theorem~\ref{theorem. main theorem} is:

\begin{corollary}
The inclusion $\dgcatt\to \Ainfty $ induces an equivalence of $\infty$-categories
	\eqnn
	\dgcatt[\eqs^{-1}] \to \Ainfty [\eqs^{-1}].
	\eqnnd
\end{corollary}

\begin{notation}
We let
	\eqnn
	\Ainftycat
	\eqnnd
denote $\Ainfty[\eqs^{-1}]$. We refer to it as the {\em $\infty$-category of $A_\infty$-categories} (over $\kk$). Implicit in the nomenclature is that all objects and morphisms are unital.
\end{notation}

In general, computing the mapping spaces of a localization such as $\Ainftycat$ is a difficult task. However, we build upon the works~\cite{gepner-haugseng, pascaleff, cos-over-rings} to prove 
$\Ainftycat$ has internal homs. 
Moreover, we prove that the $\infty$-categorical Gepner-Haugseng self-enrichment of $\Ainftycat$ 
(see Corollary~7.4.10 of~\cite{gepner-haugseng})
and the point-set Lyubashenko self-enrichment
(by dg co-categories; see Section 4 of~\cite{lyubashenko-category-of}) are compatible when one considers sufficiently cofibrant $A_\infty$-categories (see Definition~\ref{defn. homotopically projective}). This allows us to compute mapping spaces in $\Ainftycat$:

\begin{theorem}
\label{theorem. mapping spaces}
Let $A$ and $B$ be unital $A_\infty$-categories, and $A' \to A$ any quasi-equivalence from an $A_\infty$-category with homotopically projective morphism complexes. Then:
\enum[(a)]
	\item The internal hom object $\underline{\hom}_{\Ainftycat}(A,B)$ is equivalent in $\Ainftycat$ to the $A_\infty$-category 
		\eqnn
		\fun_{A_\infty}(A',B)
		\eqnnd
	of unital $A_\infty$-functors.
	\item The mapping space $\hom_{\Ainftycat}(A,B)$ is homotopy equivalent to the space
	\eqnn
	N(\fun_{A_\infty}(A',B))^{\sim}
	\eqnnd
where $N$ is the $A_\infty$-nerve and $N(-)^{\sim}$ denotes the largest $\infty$-groupoid contained in the nerve.
\enumd
\end{theorem}
In fact, we record many names for internal hom objects and mapping spaces -- see Theorems~\ref{theorem. internal homs in dgcat},
\ref{theorem. maps in dgcat},
\ref{theorem. internal homs in Aoocat}, and
\ref{theorem. mapping spaces in Aoocat}, which subsume Theorem~\ref{theorem. mapping spaces}.

\begin{remark}
It is necessary to replace $A$ by a well-behaved $A'$. For example, take $\kk = \ZZ$ and let $A$ be the $A_\infty$-category with one object and endomorphism ring $\ZZ/2\ZZ$. Take $A'$ to be a semifree resolution -- e.g., the dga $\ZZ \xrightarrow{\times 2} \ZZ$. Then $\fun(A,A')$ is empty while $\fun(A',A')$ is not. That is, functor $A_\infty$-categories are not invariant under arbitrary quasi-equivalences in the domain variable. However, $A_\infty$-categories with homotopically projective morphism complexes admit a Whitehead theorem (Proposition~\ref{prop. whitehead for a infty}), ensuring the invariance of functor categories under quasi-equivalences among such $A_\infty$-categories (Proposition~\ref{prop. pullback along equivalence is equivalence}). 
\end{remark}

\begin{remark}
The category of unital $A_\infty$-categories has no model structure (for one thing, it does not admit all equalizers -- see Section~1.5 of~\cite{cos-localizations}). Regardless, as we have hinted above, there is a good class of objects that behave as though they deserve the labels of cofibrant and fibrant:  $A_\infty$-categories with homotopically projective morphism complexes. (See Remark~\ref{defn. homotopically projective} and Section~\ref{section. whitehead style}.) The good behavior of this class of $A_\infty$-categories is exploited throughout Section~\ref{section. interhal homs and maps}.
\end{remark}

\begin{remark}
\label{remark. hp replacements}
If $A$ is a unital $A_\infty$-category, one can always find a quasi-equivalence $A' \to A$ from a unital $A_\infty$-category $A'$ with homotopically projective morphism complexes. To see this, recall that the Yoneda embedding induces a map $A \to \cY(A)$ where $\cY(A)$ is a dg-category, and the map is invertible up to natural equivalence of $A_\infty$ functors. (See the proof of Proposition~\ref{prop. homotopically projective abstract equivalences are equivalences}.) Moreover, there exists a cofibrant replacement quasi-equivalence $A' \to \cY(A)$ of dg-categories by the Tabuada model structure, and any cofibrant dg-category has homotopically projective morphism complexes. (See Notation~\ref{notation. dgcat versions}.) Letting $\cY(A) \to A$ be a homotopy inverse, the composition $A' \to \cY(A) \to A$ establishes the claim.
\end{remark}

\begin{remark}
For most models of Fukaya categories, the morphism complexes are free as graded modules. Hence most models of Fukaya categories have homotopically projective morphism complexes.
\end{remark}

\begin{remark}
Theorem~\ref{theorem. mapping spaces} contrasts with the dg case. There, even when $A$ is cofibrant, the  dg-category of dg-functors from $A$ to $B$ is rarely expected to compute the correct mapping spaces in $\dgcat$. Instead, the usual representative (as identified by To\"en~\cite{toen-homotopy-theory-of-dg-cats})  for computing the internal hom dg-category has been the dg-category of cofibrant right quasi-representable modules over $A^{\op} \tensor B$. Indeed, it seems one strong motivation for the work~\cite{cos-over-rings} was to give a representative of an internal hom object identifiable as a functor category. We build upon ibid. to indeed show that $A_\infty$ functor spaces between dg-categories computes the correct mapping space in the $\infty$-category $\dgcat$. 
(See Theorem~\ref{theorem. maps in dgcat}.)
\end{remark}

\subsection{Five equivalent $\infty$-categories}\label{section. THE infinity-category}
Let us explain the following equivalences of $\infty$-categories:
	\eqn\label{eqn. five models}
	\xymatrix{
	\dgcat
		&& \ar[ll]^{\sim}_{\cite{muro,tabuada-model-structures-on-dg-cats}} \dgcatt[\eqs^{-1}] \ar[rr]_{\sim}^{\cite{haugseng-rectification-enriched}} \ar[d]_{\sim}^{\cite{pascaleff,cos-over-rings}}
		&& \inftyCat^{\chain_k} \\
		&& \Ainfty^{\str}[\eqs^{-1}] \ar[d]_{\sim}^{\text{Thm}~\ref{theorem. main theorem}} \\
		&& \Ainfty[\eqs^{-1}].
	}
	\eqnd
First, the top row: Recall that over an arbitrary base ring $\kk$, Haugseng~\cite{haugseng-rectification-enriched} proved that the $\infty$-category of dg-categories models the $\infty$-category of $\infty$-categories enriched in chain complexes. More accurately, Haugseng proved\footnote{This is Theorem~5.8 of~\cite{haugseng-rectification-enriched}, taking ${\bf V}$ to be the projective model structure on chain complexes over a ring $\kk$. } that the localization of the 1-category $\dg$ of dg-categories, along quasi-equivalences, is equivalent to the $\infty$-category $\inftyCat^{\chain_k}$ of $\infty$-categories enriched in chain complexes over $\kk$. 

On the other hand, $\dg[\eqs^{-1}]$ admits the Tabuada model structure~\cite{tabuada-model-structures-on-dg-cats} on dg-categories\footnote{This is a special case of Muro's model structure on enriched categories~\cite{muro}.}. Above, we have used $\dgcat$ to denote a quasi-category obtained from the model structure on dg-categories. There are many ways to obtain such a quasi-category. (A tautological construction is to localize the 1-category $\dg$ along quasi-equivalences, which would make the left-pointing arrow definitionally an equality. Another is to take a simplicial localization -- for example, the hammock localization of Dwyer-Kan -- and then the nerve of the resulting simplicial category. We have no need to choose a model in this work.) In short, the top row consists of equivalences between various models all deserving to be called an $\infty$-category of dg-categories.

We now visit the middle column. It is a result of Pascaleff~\cite{pascaleff} (over any field) and Canonaco-Oranghi-Stellari~\cite{cos-over-rings} (over any ring) that the inclusion of $\dg$ into $\Ainftystr$ induces an equivalence of their $\infty$-categorical localizations along $\eqs$. (The strictness in $\Ainftystr$ is necessary to employ a bar-cobar type adjunction.)  Theorem~\ref{theorem. main theorem} removes the strict-unitality conditions on objects and functors.

The upper-right model, $\inftyCat^{\chain_k}$, is arguably the most formally well-behaved definition of $A_\infty$-categories. In contrast, the bottom two $\infty$-categories have definitions that depend on a choice of a particular model for a $\kk$-linear $E_1$ operad (e.g., taking cellular chains on a particular cellular presentation of the associahedra), and hence driven by particular choices of formulas. The equivalences above exhibit one giant compatibility check of various approaches -- $\infty$-categorical, model-categorical, and formula-driven -- that have been taken to articulate the single theory captured in the equivalence class of the above $\infty$-categories.

\begin{remark}
Some authors who work only in the stable setting (i.e., where every category is assumed to be pretriangulated) will {\em define} a $\kk$-linear stable presentable $\infty$-category to be an $\infty$-category equipped with an action of the $\infty$-category $\kk\Mod$ (where the action $\kk\Mod \times \cC \to \cC$ preserves colimits in each variable, and $\kk\Mod$ is the $\infty$-category of possibly unbounded $\kk$-linear chain complexes). Such an action endows the morphisms of $\cC$ with the structure of $\kk$-linear chain complexes. (This is most easily verified if one uses that $\kk$-linear chain complexes are the same thing as $H\kk$-linear spectra.) 

However, such a definition requires that $\cC$ has a plentiful supply of colimits, so we avoid it here.  The equivalences in~\eqref{eqn. five models} require no assumptions of presentability or (co)completeness of the $A_\infty$-categories/dg-categories in question, nor on their idempotent-completeness. Hence the collection of $A_\infty$-categories considered here is more general than those that are articulable by trying to define an $A_\infty$-category as an $\infty$-category with all colimits receiving a colimit-preserving action from $\kk\Mod$.
\end{remark}

\subsection{Applications}\label{section. applications}
\subsubsection{Avoiding specific unitalities}
Prior to this work, the community only had a proof that (a localization of) the strictly unital $\Ainfty^{\str}$ models the $\infty$-category of $A_\infty$-categories. However, in practice, natural constructions give rise to unital $A_\infty$-categories and unital functors among them (where neither the functors nor the $A_\infty$-categories need be strictly unital). For a geometric example, consider many models of Fukaya categories. For an algebraic example in the present work, consider the natural transformation $\id_{\Ainftycat} \to \tau$ in~\eqref{eqn. naturality of tau}, which is a natural tranformation of unital functors, but not of strictly unital functors. Contorting a natural, unital construction to fit into a strictly unital framework can be painful, and can now be avoided using Theorem~\ref{theorem. main theorem}. (Or, if one prefers to look under the hood, one may make use of~\eqref{eqn. naturality of tau} liberally.)

As an application, suppose one has an invariant of $A_\infty$-categories that is most easily defined for strictly unital $A_\infty$-categories, or whose functoriality is most easily proven for strictly unital functors -- so one can produce a functor $F: \Ainftystr \to \cD$, where the invariant is valued in $\cD$. If $F$ sends quasi-equivalences to equivalences in $\cD$, then it is formal that $F$ extends to a functor $\Ainftystr[\eqs^{-1}] \to \cD$. 
Theorem~\ref{theorem. main theorem} allows us to formally conclude that (an extension of) the invariant $F$ can also be defined homotopy-coherently for all (non-strictly) unital $A_\infty$-categories and all (non-strictly) unital functors. Perhaps a more dramatic application is that, by the middle column of~\eqref{eqn. five models}, a quasi-equivalence-invariant $F$ defined on dg-categories (rather than strictly unital $A_\infty$-categories) also canonically extends to a homotopically well-behaved invariant of all unital $A_\infty$-categories.

Even better, by the universal property of localizations, Theorem~\ref{theorem. main theorem}  shows that such an extension is unique up to contractibly canonical equivalence. Thus, the equivalence of two invariants (not just on individual $A_\infty$-categories, but as invariants sensitive to spaces of functors) can be checked simply by comparing their values on strictly unital $A_\infty$-categories (or dg-categories) and strictly unital functors (or dg functors). Two salient examples are Hochschild chains (as a functor from $\Ainftycat$ to the $\infty$-category of chain complexes) or $K$-theory (as a functor from $\Ainftycat$ to the $\infty$-category of spectra -- this functor factors, by definition, through the pretriangulated completion operation).

\subsubsection{A coherent setting for inverting quasi-equivalences}
Presenting the $\infty$-category of $A_\infty$-categories as a localization is also useful when contemplating invariants that are valued in $A_\infty$-categories. As an example, some natural constructions involving Fukaya categories take place in a setting where one must formally invert quasi-equivalences. (See the opening paragraphs of Section~2.2 in~\cite{gps-descent}, diagram~(A.19) of ibid., and the discussion preceding~(11.12) of ibid.)

In particular, when claims are made of certain maps factoring -- as in the stop-removal-is-localization result (Lemma~3.12 of ibid.) -- one must specify in what category the maps factor. In the $\infty$-category of $A_\infty$-categories, the inversion of the quasi-equivalence in (3.5) of ibid. may be performed homotopy-coherently, and contractibly-uniquely, allowing for the desired factorization up to canonical homotopy. Importantly, if one is to remove more than one (component of) a stop successively, a claim at the level of the homotopy category of $A_\infty$-categories is not sufficient to conclude that the diagram relating successive stop-removals and localizations is homotopy coherent. In the $\infty$-category of $A_\infty$-categories, however, this coherence follows automatically. 

Combined with the localization results of~\cite{last-1}, one can in fact conclude not only the coherence of successive stop removals, but their equivariance and continuity with respect to the relevant actions of self-embedding spaces of sectors.

\subsubsection{Mapping spaces and Hochschild cochains}
Finally, Theorem~\ref{theorem. mapping spaces} shows that the above applications give rise to desirable properties regarding mapping spaces. A functor out of (or to) $\Ainftycat$ indeed gives rise to coherent maps out of (or to) spaces of functors. (This is also a satisfying coda to a well-known problem for dg-categories: In the infinity-category of dg-categories, a morphism does not always admit an easy interpretation as a dg functor -- To\"en's resolution instead invited us to examine cofibrant right quasirepresentable bimodules. The paper~\cite{cos-over-rings} demonstrates an improvement by instead being able to think of a map in $\dgcat$ as an $A_\infty$ functor between dg-categories.) As a final point, we note that by combining our two main results, a not-strictly-unital $A_\infty$-functor between two strictly unital $A_\infty$-categories can, canonically, be homotoped to a strictly unital one (so long as the domain $A_\infty$-category has homotopically projective mapping complexes).

To illustrate the utility of having characterizations of mapping spaces, we present two immediate applications of Theorem~\ref{theorem. mapping spaces}. Proofs are given at the end of the paper. Note also that for any two $A_\infty$-categories $A,B$, the unital functors $A \to B$ form the objects of an $A_\infty$-category $\fun_{A_\infty}(A,B)$. We take the cohomology groups of its morphism complexes repeatedly in the following:

\begin{corollary}\label{cor. cohomology and homotopy groups}
Let $A'$ be a unital $A_\infty$-category with homotopically projective mapping complexes. Let $B$ be any unital $A_\infty$-category. Then
\enum[(a)]
\item
	\eqnn
	\pi_{0}\hom_{\Ainftycat}(A',B)
	\cong
	\{\text{unital functors $f: A ' \to B$} \}/\sim
	\eqnnd
where the equivalence relation is homotopy (i.e., natural transformations that are invertible). That is, $\pi_0$ of mapping spaces in $\Ainftycat$ are computed as homotopy classes of unital $A_\infty$-functors.
\item Now suppose $f,g: A' \to B$ are homotopic functors (i.e., related by natural equivalence). Then 
	\eqnn
	\pi_{1}(\hom_{\Ainftycat}(A',B), f)
	\cong
	\left(H^{0}\hom_{\fun_{A_\infty}(A', B)}(f,g)\right)^\times
	\eqnnd
where the superscript $\times$ indicates the (cohomology classes of) the homotopy-invertible natural transformations from $f$ to $g$. 
\item
And for all $i \geq 1$,
	\eqnn
	\pi_{i+1}(\hom_{\Ainftycat}(A',B), f)
	\cong
	H^{-i}\hom_{\fun_{A_\infty}(A', B)}(f,g).
	\eqnnd
That is, the negative cohomology groups of the complex of natural transformations are the higher homotopy groups of the mapping spaces in $\Ainftycat$.

\item
Further, choosing an $A_\infty$-category $A'$ with homotopically projective morphism complexes, we have
	\eqnn
	\pi_1 ( (\Ainftycat)^\sim,A')
	\cong 
	(\pi_{0}\hom_{\Ainftycat}(A',A'))^\times
	\cong
	(\{\text{unital functors $f: A ' \to A'$} \}/\sim)^{\times}
	\eqnnd
where the superscript $\times$ indicates we are only considering those endofunctors that are, up to natural equivalence, invertible.
\item And
	\eqnn
	\pi_2 ( (\Ainftycat)^\sim,A')
	\cong 
	\pi_{1}(\hom_{\Ainftycat}(A',A'),\id_A)
	\cong
	\left(H^{0}\hom_{\fun_{A_\infty}(A', B)}(f,g)\right)^\times.
	\eqnnd
\item Finally, for all $i \geq 3$, we have
	\eqnn
	\pi_i ( (\Ainftycat)^\sim,A')
	\cong
	\pi_{i-1}(\hom_{\Ainftycat}(A',B), f)
	\cong
	H^{-i+2}\hom_{\fun_{A_\infty}(A', B)}(f,g).
	\eqnnd
\enumd
\end{corollary}

Our final application concerns Hochschild cohomology of an $A_\infty$-category. Let us state at the outset that there are two natural starting points for the definition of the Hochschild cochain complex of an $A_\infty$-category $A$. In either starting point, it is our desire to avoid presentations of Hochschild cochains that are rooted only in chain-complex-dependent formulas, so that universal characterizations are available to us.

First, one could define the notion of an $(A,A)$-bimodule -- which one most often defines as a functor $A^{\op} \tensor A \to \chain$ -- and prove that the hom pairing is an example of a bimodule. (One often calls this the diagonal bimodule.) Then one can define the Hochschild cochain complex to be the derived endomorphisms of the diagonal bimodule in the $A_\infty$-category of bimodules. A coend construction (if one can construct a theory of coends for $A_\infty$-categories) realizes a monoidal structure on the $A_\infty$-category of $(A,A)$-bimodules, with the diagonal bimodule as the unit. By Dunn addivitiy, this even exhibits the $E_2$-algebra structure on this definition of Hochschild cochains. However, as far as we know, there is no satisfactory $\infty$-categorical framework for enriched $\infty$-categories incorporating coends and an enriched Yoneda pairing. So the approach of this paragraph, with present technology, would require new inventions. 

Another definition of Hochschild cochains of $A$ is as the endomorphisms of the identity functor. If one exhibits $\fun_{A_\infty}(A,A)$ as a monoidal $A_\infty$-category under composition, the identity functor is clearly the unit, and this exhibits an $E_2$-algebra structure on Hochschild cochains of $A$, again by Dunn additivity. The approach of the present paragraph is already available to us: Theorem~\ref{theorem. mapping spaces} exhibits the self-enrichment of $\Ainftycat$ (induced by internal hom objects) as precisely modeled by functor $A_\infty$-categories for  objects with homotopically projective morphism complexes. This in particular exhibits the composition monoidal structure on internal endomorphisms.
Hence, if $A$ has homotopically projective morphism complexes, we obtain a workable definition of Hochschild cochains of $A$:
	\eqn\label{eqn. HH definition}
	CH(A) :=
	\hom_{\underline{\hom}_{\Ainftycat}(A,A)}(\id_A,\id_A)
	\simeq
	\hom_{\fun_{A_\infty}(A,A)}(\id_A,\id_A).
	\eqnd
We take this to be our definition of Hochschild cochains. When one has a concrete model for $A$, the last equivalence above produces a cochain complex model. (And when $A$ lacks homotopically projective morphism complexes, one may replace $A$ with a quasi-equivalent $A'$ having homotopically projective morphism complexes to compute $CH(A)$. This is always possible by Remark~\ref{remark. hp replacements}.)

Let us remark that the the two approaches above should yield equivalent answers because any theory of bimodules will allow for a natural map $\underline{\hom}_{\Ainftycat}(A,B) \to \bimod(A,B)$ that is fully faithful (and the essential image would be the so-called {\em graphical} bimodules). 

Taking~\eqref{eqn. HH definition} to be our definition of Hochschild cochains (well-defined up to quasi-isomorphism) and defining Hochschild cohomology $HH^*$ to be the cohomology of~\eqref{eqn. HH definition}, we have the following. One can view it as an $A_\infty$-version of Corollary 8.3 of~\cite{toen-homotopy-theory-of-dg-cats}. 
We refer the reader to Theorem~4.5 of~\cite{faonte-2} for another proof in the case $\kk$ is a field.

\begin{corollary}\label{cor. HH}
Let $A'$ be a unital $A_\infty$-category with homotopically projective morphism complexes. 
Letting $HH^0(A')^\times$ denote the multiplicatively invertible elements of $HH^0(A')$, we have
	\eqnn
	\pi_2 ( (\Ainftycat)^\sim,A')
	\cong 
	\pi_{1}(\hom_{\Ainftycat}(A',A'),\id_A)
	\cong
	(HH^0(A'))^\times.
	\eqnnd
For all $i \geq 1$, we have
	\eqnn
	\pi_{i+2} ( (\Ainftycat)^\sim,A')
	\cong
	HH^{-i}(A').
	\eqnnd
\end{corollary}

\subsubsection{Universal properties of quotients and localizations}
By virtue of our mapping space computations, we can now verify that constructions of localizations and quotients in the literature indeed satisfy the $\infty$-categorical universal properties that the names imply. This is the subject of~\cite{oh-tanaka-localizations}.

\subsection{Acknowledgments}
We thank 
Alberto Canonaco,
Mattia Ornaghi, and
Paolo Stellari for very helpful feedback on earlier drafts, catching some mistakes, and asking clarifying questions that greatly improved the exposition. We also thank them for sharing with us a different proof of Lemma~\ref{lemma. right inverse} (their proof will appear in ~\cite{cos-over-rings}) and for greatly simplifying the proof of Lemma~\ref{lemma. j has left inverse} (their simplification now appears here; the more computationally involved, old proof has been relegated to the appendix of this work). The author was supported by a Sloan Research Fellowship and by an NSF CAREER Award during the writing of this work.

\section{Unitalities}

\subsection{A fact about localizations of infinity-categories}

\begin{recollection}
\label{recollection. localizations}
If $\cC$ is a category and $W$ is some collection of morphisms in $\cC$ (or, equivalently, the subcategory generated by the collection), the notation
	$
	\cC[W^{-1}]
	$
denotes the $\infty$-categorical localization of $\cC$ along $W$. Informally, $\cC[W^{-1}]$ is equipped with a functor $\cC \to \cC[W^{-1}]$, and this functor exhibits $\cC[W^{-1}]$ as the initial $\infty$-category in which the morphisms in $W$ become invertible. Note that $W$ and $W'$ necessarily define the same localization if they define the same homotopy classes of morphisms. In particular, one may specify a localization by specifying homotopy classes of morphisms in $\cC$. (For a point-set model, one takes $W$ to consist of all morphisms in $\cC$ contained in these homotopy classes). 

There are many models for this localization. The reader may take $\cC[W^{-1}]$ to be the quasi-category obtained as the homotopy pushout of quasi-categories
	\eqnn
	\xymatrix{
	N(W) \ar[d] \ar[r] & N(\cC) \ar[d] \\
	|N(W)| \ar[r] & \cC[W^{-1}].
	}
	\eqnnd
Here, $N(\cC)$ is the nerve of $\cC$,  $|X|$ is a Kan-fibrant replacement of the simplicial set $X$, and $\cC[W^{-1}]$ may be modeled by taking a set-theoretic pushout of simplicial sets, then Joyal-fibrantly replacing the result. 

In practice, one often utilizes the model-categorical power of {\em marked} simplicial sets as developed in Section~3.1 of~\cite{htt} (and, in the notation of loc. cit., one often gets away with just setting the base simplicial set $S$ to be a point, e.g., $S=\Delta^0$). Further using the notation of Definition~3.1.0.1 of~\cite{htt}, the localization $\cC[W^{-1}]$ is a fibrant replacement of the pair $(X,\cE)$ where $\cE$ is the collection of edges in $W$.
\end{recollection}

\begin{recollection}
\label{recollection. localization respects products}
There is an $\infty$-category $\winftyCat$ (Construction~4.1.3.1 of~\cite{higher-algebra}) whose objects are pairs $(\cC,W)$ consisting of an $\infty$-category $\cC$ and a (not necessarily full) subcategory $W$ of the homotopy category $\ho \cC$ containing all isomorphisms in $\ho \cC$. (One may equivalently think of $W$ as the maximal sub-simplicial-set of $\cC$ whose image in $\ho \cC$ is the subcategory in question.)

The space of morphisms $\hom_{\winftyCat}( (\cC,W),(\cC',W'))$ is identified with the connected components of $\hom_{\inftyCat}(\cC,\cC')$ spanned by those functors for whom (after passing to homotopy categories) the image of $W$ is contained in $W'$. There is a functor 
	\eqnn
	\LL: \winftyCat \to \inftyCat
	\eqnnd
which on objects sends $(\cC,\cW)$ to the localization $\cC[\cW^{-1}]$, and which commutes with products (Proposition~4.1.3.2 of~\cite{higher-algebra}). 
\end{recollection}

Fix a functor $\eta: \Delta^1 \times \cC \to \cC'$. We let $f_0$ and $f_1$ denote the restrictions of $\eta$ to $\{0\} \times \cC$ and to $\{1\} \times \cC$, respectively (so that $\eta$ is a natural transformation from $f_0$ to $f_1$).

\begin{prop}
\label{prop. good natural transformations induce homotopies}
Suppose that $\eta$ (after passing to homotopy categories) sends $\Delta^1 \times W$ inside $W'$. Then $\eta$ induces a homotopy (i.e., a natural equivalence -- not just a natural transformation) from $\LL(f_0)$ to $\LL(f_1)$.
\end{prop}

\begin{proof}
$\eta$ defines a morphism in $\winftyCat$
	\eqnn
	(\Delta^1 \times \cC, \Delta^1 \times W) \to (\cC',W')
	\eqnnd
hence a morphism in $\inftyCat$
	\eqnn
	\LL \eta: \Delta^1[(\Delta^1)^{-1}] \times \cC[W^{-1}] \simeq
		(\Delta^1 \times \cC) [ (\Delta^1 \times W)^{-1}] \xrightarrow{\LL \eta}
		\cC'[(W')^{-1}].
	\eqnnd
(The first equivalence uses the fact that $\LL$ commutes with products -- see Recollection~\ref{recollection. localization respects products}). Of course, the localization $|\Delta^1| := \Delta^1[(\Delta^1)^{-1}]$ inverts the unique non-degenerate edges $0 \to 1$ in $\Delta^1$ and receives a map $\Delta^1 \to |\Delta^1| $ that is a bijection on objects. In particular, $\LL \eta$ exhibits an invertible-up-to-homotopy natural transformation (and all the data exhibiting the homotopy-invertibility) from $\LL f_0$ to $\LL f_1$.
\end{proof}

\subsection{Notation, conventions, and useful examples}

We assume the reader is familiar with $A_\infty$-categories. We include this section to set   notation and  sign conventions. 

\begin{notation}[$s$]
We let $s$ be the shift operator on chain complexes and on graded $\kk$-modules, so that $s\kk$ is a copy of $\kk$ in cohomological degree -1. For example, the mapping cone of a chain map $f: A \to B$ decomposes, as a graded $\kk$-module, as $B \oplus sA$. As another example, we have $\hom(A,sB) \cong s\hom(A,B) \cong \hom(s^{-1}A,B)$ as chain complexes. Note $s$ is the chain-complex analogue of the reduced suspension $\Sigma$ of pointed spaces and of spectra.
\end{notation}

\begin{convention}[Signs for $A_\infty$-operations]
Given integers $\alpha,\beta, \gamma$, define
	\eqn\label{eqn. sign convention star}
	\star = \star(\alpha,\beta,\gamma) := \alpha+\beta\gamma.
	\eqnd
Assume we have $k$-ary operations $m^k$ of (cohomological) degree $2-k$. Given two positive integers $k \leq l$, we define
	\eqnn
	m^k_l := \sum (-1)^{\star} \id^{\tensor \alpha} \tensor m^\beta \tensor \id^{\tensor \gamma}
	\eqnnd
where the summation runs through $\alpha,\gamma \geq 0$ and $\beta \geq 1$ satisfying $\alpha+\beta+\gamma = l$ and $\beta = k$. Note $m^k_k = m^k$. We say that the collection of $k$-ary operations $\{m^k\}_{k \geq 1}$ satisfies the {\em $A_\infty$ relations} if for all $l \geq 1$,
	\eqn\label{eqn. Aoo relations}
	\sum_{1 \leq k \leq l} m^{l-k+1}m^k_l = 0.
	\eqnd
\end{convention}

\begin{remark}[Relating some sign conventions]
All $A_\infty$ sign conventions depend on a choice of the function $\star$. Our convention follows that of Getzler-Jones~\cite{getzler-jones-cyclic} and Keller~\cite{keller-intro-to-algebras-modules} -- this is the standard convention that is forced upon us if we choose to formula the Koszul dual notion of dg-coalgebra as satisfying a quadratic equation with no signs, and if we employ the Koszul sign rule. Another convention is that used by Lyubashenko and coauthors -- see (2.3.2) of~\cite{lyubashenko-category-of} -- and in works such as~\cite{cos-over-rings}. The sign difference in these two conventions owes to right-versus-left ordering of compositions, and one can easily work out that the sign differences correspond to reversing orientations of cells in associahedra. 
\end{remark}

\begin{example}
For $l=1,2,3$ the equation~\eqref{eqn. Aoo relations} becomes
	\begin{align}
	0& = m^1 m^1  \nonumber \\
	0& = m^1m^2 - m^2(m^1 \tensor \id + \id \tensor m^1)  \nonumber \\
	0& = m^1m^3 + m^3(m^1 \tensor \id \tensor \id + \id \tensor m^1 \tensor \id + \id \tensor \id \tensor m^1) + 
	m^2(m^2 \tensor \id - \id \tensor m^2 )
	 \nonumber .
	\end{align}
Thus $m^1$ is a differential for a cochain complex. $m^2$ is a chain map (i.e., satisfies the Leibnix rule). And $m^3$ exhibits a homotopy from $m^2(m^2 \tensor \id)$ to $m^2(\id \tensor m^2)$. 
\end{example}

\begin{recollection}[Augmentations and naturality]
\label{recollection. augmentation}
Given an $A_\infty$-category $A$, one has an augmented $A_\infty$-category $A^+$, obtained by formally adjoining strict units to $A$. Concretely, we have 
	\eqnn
	\hom_{A^+}(X,Y)
	:= 
	\begin{cases}
	\hom_A(X,Y) & X \neq Y \\
	\hom_A(X,X) \oplus \kk & X = Y.
	\end{cases}
	\eqnnd
We will often call the adjoined strict units the augmentation units, and denote the augmentation unit of an object $X \in \ob A$ by $1_X$. Any $A_\infty$-functor $f: A \to B$ induces a strictly unital functor $f^+: A^+ \to B^+$. One has $\id_A^+ = \id_{A^+}$ and $(f \circ g)^+ = f^+ \circ g^+$.

The (non-unital) inclusion $A \to A^+$ is natural: For any $f$, the compositions $A \to A^+ \to B^+$ and $A \to B \to B^+$ are equal.
\end{recollection}

\begin{recollection}[Twisted complexes]
Given an $A_\infty$-category $A$, we let $\Tw A$ denote the $A_\infty$-category of twisted complexes (of objects in $A$). We refer to Section 7 of~\cite{keller-intro-to-algebras-modules} for details and for one sign convention. In~\cite{bespalov-lyubashenko-manzyuk}, the notation $A^{\mathsf{tr}}$ is used in place of $\Tw A$.
\end{recollection}

\begin{example}[Morphisms between cones are cones]
If $U$ and $T$ are cochain complexes with a chain map $f: U \to T$, the cone $\cone(f)$ is a cochain complex $T \oplus sU$ with differential
	\eqnn
	t \oplus su \mapsto (dt  \pm fu )\oplus \pm sdu.
	\eqnnd
There are multiple sign conventions for what one means by the mapping cone. Even if one demands the Koszul-sign inspired identity $d(su) = - s(du)$, either of the differentials
	\eqnn
	(dt - fu) \oplus -sdu,
	 \qquad
	(dt + fu) \oplus -sdu
	\eqnnd
leads to a fine theory of cones in homological algebra. (For example, the map $t \oplus su \mapsto t \oplus -su$ realizes a natural chain complex isomorphism between the two mapping cones.) Yet another differential for the mapping cone is $(dt + (-1)^u fu) \oplus sdu$, and one can check this model for the mapping cone is isomorphic to the previous two. We will refer to a chain complex isomorphic to any of these as a mapping cone for $f$, though we may take certain isomorphisms for granted and say ``the'' mapping cone from time to time.

We will often make use of the following: For any closed degree zero morphism $g$ in $A$, the morphism complexes $\hom_{\Tw A}(X,\cone(g))$ are also mapping cones for any $X \in \Tw A$. 
\end{example}

\begin{example}
\label{example. hom of cones is a cone}
To see an example of the above claim, fix an object $X$ and a unit $e \in \hom_A(W,W)$. Then $\hom_{\Tw A}(\cone(e),X)$ and $\hom_{\Tw A}(X,\cone(e))$ are both acyclic -- in fact, they are mapping cones of a quasi-isomorphism. For example, 
	\eqnn
	\hom_{\Tw A}(X, \cone(e)) = \hom_A(X,W) \oplus s\hom_A(X,W)
	\eqnnd
with differential
	\eqnn
	m^1_{\Tw A} : p \oplus sq \mapsto ( m^1 p  \pm m^2(e,q) )\oplus - sm^1(q).
	\eqnnd
We recognize this as the differential for a mapping cone of the chain map $m^2(e,-)$.
Likewise, $\hom_{\Tw A}(\cone(e), X)$ is a (shift of) the mapping cone for the map $m^2(-,e)$.
Because $e$ is a unit, each mapping complex is a mapping cone of a homotopy-invertible chain map. Hence $\hom_{\Tw A}(X, \cone(e))$ and $\hom_{\Tw A}(\cone(e), X)$ are acyclic.
\end{example}

\begin{example}
Suppose we have two degree 0, closed elements $f \in \hom_A(X,Y)$ and $g \in \hom_A(X',Y')$. Then we can represent $\hom_{\Tw A}( \cone(f),\cone(g))$ as a matrix of graded $\kk$-modules:
	\eqnn
	\left(
	\begin{array}{cc}
	\kk \tensor \hom_A(Y,Y') & s^{-1}\kk \tensor \hom_A(X,Y) \\
	s\kk \tensor \hom_A(Y,X') & \kk \tensor \hom_A(X,X').
	\end{array}
	\right)
	\eqnnd
Writing $1$ for the unit of $\kk$, we denote the generator of $s^{-1}\kk$ as $s^{-1}1$ and likewise we write $s1 \in s\kk$. Then $m^1_{\Tw A}$ of the element
	\eqn\label{eqn. matrix elements}
	\bx = 
	\left(
	\begin{array}{cc}
	p & s^{-1} 1 \tensor x \\
	s 1 \tensor  q & y
	\end{array}
	\right)
	\in \hom_{\Tw A}( \cone(f),\cone(g))
	\eqnd
is given up to sign by
	\eqn\label{eqn. differential of cone cone hom}
	m^1_{\Tw A} \bx = 
	\left(
	\begin{array}{cc}
	m^1p \pm  m^2(g,q) 
		&  s^{-1}1 \tensor 
			(
			\pm m^1 x \pm  m^2(p,f)  \pm  m^2(g,y) \pm  m^3(g,q,f)
			)\\
	s1 \tensor (\pm  m^1 q) & m^1 y  \pm  m^2(q,f)
	\end{array}
	\right)	
	\eqnd
\end{example}

\begin{example}
Fix two units $e,e'$ for objects $W,W'$, respectively. We claim $\hom_{\Tw A}(\cone(e),\cone(e'))$ is acyclic -- it will again be the mapping cone of a homotopy-invertible map.

To see this, we note that (as a graded $\kk$-module) this complex is a tensor product of a matrix algebra $\End(\kk \oplus s\kk)$ with $\hom_A(W,W')$:
	\eqnn
	\hom_{\Tw A}(\cone(e),\cone(e'))
	\cong
    \left(
    \begin{array}{cc}
    \hom(\kk,\kk)  &  \hom(s\kk,\kk) \\
    \hom(\kk,s\kk) &  \hom(s\kk,s\kk) \\
    \end{array}
    \right)
    \tensor
    \hom_A(W,W').
	\eqnnd
We fix elements of this matrix algebra 
	\eqn\label{eqn. a b c d tensors}
	a=
        \left(
        \begin{array}{cc}
        1 & 0 \\
        0 & 0 \\
        \end{array}
        \right), \qquad
	b=
        \left(
        \begin{array}{cc}
        0 & 1 \\
        0 & 0 \\
        \end{array}
        \right), \qquad
	c=
        \left(
        \begin{array}{cc}
        0 & 0 \\
        1 & 0 \\
        \end{array}
        \right), \qquad
	d=
        \left(
        \begin{array}{cc}
        0 & 0 \\
        0 & 1 \\
        \end{array}
        \right).
	\eqnd
(Here, $b$ is the degree 1 isomorphism sending $1 \in s \kk$ to $1 \in \kk$. Likewise, $c$ is a degree -1 isomorphism, while $d$ is a degree 0 isomorphism.)
Then any element in $\hom_{\Tw A}(\cone(e),\cone(e'))$ can be written
	\eqn\label{eqn. abcd sum tensor}
	a \tensor p + b \tensor x + c \tensor q + d \tensor y \in \hom_{\Tw A}(\cone(e),\cone(e'))
	\eqnd
for a unique quadruplet $p,q,x,y \in \hom_A(W,W')$.

Using the differential~\eqref{eqn. differential of cone cone hom}, we see that $\hom_{\Tw A}(\cone(e),\cone(e'))$ is (a shift of) the mapping cone of the endomorphism
	\eqn\label{eqn. cone cone differential}
	Z =
	\left(
	\begin{array}{c}
	m^2(-,e) \pm m^3(e',\bullet,e) \\
	m^2(\bullet,e)
	\end{array}
	\right)
	\in
	\left(
	\begin{array}{c}
	\hom(\kk,\kk)  \tensor \hom_{\Tw A}(W,W') \\
	\hom(\kk,s\kk) \tensor  \hom_{\Tw A}(W,W')
	\end{array}
	\right).
	\eqnd
We claim that $Z$ admits a homotopy inverse, namely the map
	\eqnn
	W = 
	\left(
	\begin{array}{c}
	m^2(-,e_1) \\ m^2(\bullet,e_1)
	\end{array}
	\right).
	\eqnnd
To see this, observe the homotopies of operators
	\begin{align}
	m^2(m^3(e_2,\bullet,e_1),e_1)
	&\sim
		- m^2(e_2,m^3(\bullet,e_1,e_1))
		+ m^3(m^2(e_2,\bullet),e_1,e_1)
		\nonumber \\ 
		& \qquad
		- m^3(e_2,m^2(\bullet,e_1),e_1)
		+ m^3(e_2,\bullet,m^2(e_1,e_1)) \nonumber \\
	&\sim
		-m^3(\bullet,e_1,e_1)
		+ m^3(\bullet,e_1,e_1) \nonumber \\
		&\qquad
		- m^3(e_2,\bullet,e_1)
		+ m^3(e_2,\bullet,e_1) \nonumber \\
	&= 0. \label{eqn. homotopy in cone cone}
	\end{align}
The first homotopy is due to the $A_\infty$ relation for $m^4(e_2,\bullet ,e_1,e_1)$ and the fact that the $e_i$ are closed. The second homotopy uses the assumption that $e_i$ are units, so that $m^2(e_2,\bullet)$ and $m^2(\bullet,e_1)$ are homotopic to the identity map.
We then have that
	\begin{align}
		WZ
	\left(
	\begin{array}{c}
	- \\ \bullet
	\end{array}
	\right) 
	&= 
    	\left(
    	\begin{array}{c}
    	m^2(m^2(-,e_1)) \pm m^2(m^3(e_2,\bullet ,e_1),e_1) \\ m^2(m^2(\bullet,e_1),e_1)
    	\end{array}
    	\right)  \nonumber \\
	&\sim
    	\left(
    	\begin{array}{c}
    	m^2(-,e_1) \pm 0  \\ m^2(\bullet,e_1)
    	\end{array}
    	\right) \nonumber \\
	&\sim
    	\left(
    	\begin{array}{c}
    	-  \\ \bullet
    	\end{array}
    	\right) 
	\end{align}
where the first homotopy uses~\eqref{eqn. homotopy in cone cone} and the fact that $e_1$ is a unit. The last homotopy uses that $e_1$ is a unit again. A similar argument shows that $ZW \sim \id$, showing that $Z$ is homotopy invertible.
\end{example}

\begin{recollection}[Localizations and naturality]
\label{recollection. localizations of Aoo cats}
Given a collection of closed degree zero morphisms $I \subset A$, one can form the localization $A[I^{-1}]$. We model this following~\cite{gps-covariant}: An object of $A[I^{-1}]$ is an object of $A$, while each morphism complex 
	\eqnn
	\hom_{A[I^{-1}]}(X,Y)
	\eqnnd
is defined as a bar construction, which as a graded $\kk$-module decomposes as
	\eqn\label{eqn. bar complex homs}
	\bigoplus_{l \geq 1} 
		\bigoplus_{f_1,\ldots,f_{l-1}} s^{l-1}\left(\hom_{\tw A}(C_{l-1},Y) \vert \hom_{\tw A}(C_{l-2},C_{l-1})   \vert  \ldots  \vert   \hom_{\tw A}(C_1,C_2)  \vert \hom_{\tw A}(X,C_1)  \right).
	\eqnd
Here, we have used the vertical bar $\vert$ in place of a tensor product over $\kk$. 
Each $C_i:= \cone(f_i)$ is a mapping cone for some morphism $f_i$ in $I$; as indicated in the subscripts of $\hom$, we treat these mapping cones as objects of $\tw A$. The $l=1$ summand is simply $\hom_A(X,Y) = \hom_{\tw A}(X,Y)$. The above morphism complex is filtered by word-length, $l$. Given objects $X_0,\ldots,X_N$ of $A$, the operation $m^N_{A[I^{-1}]}$  on the sub-complex of words of length $l_N,\ldots,l_1$ is a summation 
	\eqn\label{eqn. operations in bar construction}
	\sum (-1)^{\ast} \id^{\tensor \alpha} \tensor m^{\beta}_{\tw A} \tensor \id^{\tensor \gamma}
	\eqnd
where $0 \leq \alpha \leq l_N-1$, $0\leq \gamma \leq l_1-1$, and $\alpha+\beta+\gamma = l_1 + \ldots + l_N$. 
The sign $(-1)^\ast$ depends on further conventions in one's model of $A_\infty$-categories. For one convention, we refer to Lyubashenko-Ovsienko~\cite{lyubashenko-ovsienko}, where the above formula is presented in cocategory form in equation (2.2.1). (Using the notation of ibid., set $\cC = \Tw A$ and set $\cB = \cone(I)$ to be the full subcategory of $\cC$ consisting of mapping cones of elements in $I$. One obtains the quotient category $\sD(\cC | \cB)$ in ibid., and the localization $\cC[I^{-1}]$ is the full subcategory of  $\sD(\cC | \cB)$ spanned by objects of $A$.) 

The inclusion of the word-length 1 morphisms defines a functor $A \to A[I^{-1}]$. 
If $f:A \to B$ sends elements of $I_A$ to elements of $I_B$, then one has an induced functor $A[I_A^{-1}] \to B[I_B^{-1}]$, and this construction respects composition (of functors respecting $I$). 

Localization is natural: If $f:A \to B$ sends elements of $I_A$ to elements of $I_B$, then $A \to B \to B[I_B^{-1}]$ and $A \to A[I_A^{-1}] \to B[I_B^{-1}]$ are the same map (Corollary~3.2 of~\cite{lyubashenko-ovsienko}).
\end{recollection}

We will need the following notion:

\begin{defn}
\label{defn. contractible}
Let $A$ be an $A_\infty$-category and $B \subset A$ a full subcategory. We say that a functor $f: A \to E$ is {\em contractible along $B$} if for every $X,Y \in \ob B$, the chain map
	\eqnn
	\hom_A(X,Y) \to \hom_E(fX,fY)
	\eqnnd
is null-homotopic. 
\end{defn}

\subsection{A right inverse}

\begin{notation}($\id_A$ and $\tau$)
Given a unital $A_\infty$-category $A$, let $\id_A$ denote the collection of units in $A$. Then $\id_A$ is also a collection of closed, degree-zero morphisms in $A^+$. We set
	\eqnn
	\tau(A) := \Aploc .
	\eqnnd
That is, $\tau(A)$ is a localization of $A^+$ along the units of $A$.
Any unital functor $f: A \to B$ induces a strictly unital map
	\eqnn
	\tau(f): \Aploc \to B^+[\id_B^{-1}].
	\eqnnd
We thus obtain a functor (in the classical sense)
	\eqnn
	\tau: \Ainfty \to \Ainftystr,
	\qquad A \mapsto \Aploc,
	\qquad
	f \mapsto \tau(f).
	\eqnnd
\end{notation}

By naturality of augmentation and localization, any unital functor $f: A \to B$ fits into a commutative diagram of $A_\infty$-categories
	\eqn\label{eqn. naturality of tau}
	\xymatrix{
	A \ar[r] \ar[d]^f
		& A^+ \ar[r] \ar[d]^{f^+}
		& \Aploc \ar[d]^{\tau(f)} \\
	B \ar[r]
		& B^+ \ar[r]
		& B^+[\id_B^{-1}].
	}
	\eqnd
	
The main goal of this section is to prove:

\begin{lemma}
\label{lemma. right inverse}
For any unital $A$, the inclusion $A \to \Aploc$ is a quasi-equivalence of $A_\infty$-categories.
\end{lemma}

\begin{remark}
After we circulated the present pre-print, Canonaco-Ornaghi-Stellari found a different proof of the above result. Theirs will appear in~\cite{cos-over-rings}.
\end{remark}

\begin{remark}
We caution that $\Aploc$ is a localization of $A^+$ along the units of $A$ -- not along the units of $A^+$. In particular, when $A$ is unital, the map $A^+ \to \Aploc$ is never a quasi-equivalence. To see this, if $X$ is an object of $A$ and $e_X$ is a unit in $X$, we note that $1_X$ and $e_X$ are not cohomologous in $\hom_{A^+}(X,X)$. However, because $e_X$ is homotopy-idempotent in $A$,  $e_X$ becomes a unit of $X$ in the localization -- in particular, $[1_X] = [e_X] \in H^0_{\Aploc}(X,X)$. (If $e_X$ is a strict unit, the reader may easily verify this by using the word $1|(1_X-e_X)$ in the localization. When $e_X$ is a unit, one must also utilize the element realizing the cohomology-level equality $[m^2(e_X,e_X)] =[e_X]$.) This in particular shows that the map $A^+ \to \Aploc$ is not even an injection on cohomology.

For a concrete example, the reader may wish to contemplate the example of $A$ having one object $X$ whose endomorphism complex is the base ring $\kk$. 
\end{remark}

\begin{remark}
We think of $\tau(A)$ as a ``tautologous'' (hence the $\tau$) strict replacement for $A$.
\end{remark}

\begin{remark}
\label{remark. right inverse}
Assume Lemma~\ref{lemma. right inverse}. Then the map $A \to \Aploc$ is unital (because it is a quasi-equivalence). Further, the commutativity of~\eqref{eqn. naturality of tau} tells us that if $f: A \to B$ is a quasi-equivalence, then so is $\tau(f)$. In particular, $\tau$ passes to the localization:
	\eqnn
	\Ainfty[\eqs^{-1}]
	\to \Ainfty^{\str}[\eqs^{-1}].
	\eqnnd
The commutativity of~\eqref{eqn. naturality of tau} also tells us that the inclusion $A \to A[\id_A^{-1}]$ defines a natural transformation $\eta: \id_{\Ainfty} \to j \circ \tau$, where $j: \Ainftystr \into \Ainfty$ is the inclusion. Passing to the localization $\Ainfty[\eqs^{-1}]$, Lemma~\ref{lemma. right inverse} and Proposition~\ref{prop. good natural transformations induce homotopies} imply that the induced natural transformation $\LL \eta$ is in fact a natural equivalence. Thus, the functor induced by $j$ has a homotopy right inverse: the functor induced by $\tau$. 
\end{remark}

Before proving Lemma~\ref{lemma. right inverse}, we gather two preliminary results (Propositions~\ref{prop. homotopic cones} and~\ref{prop. localizing cohomologous}). They are intuitively clear (and useful!) statements whose proofs we could not locate in the literature, so we record them here as they may be of independent interest to the $A_\infty$-category user.

\begin{prop}
\label{prop. homotopic cones}
Let $A$ be a unital $A_\infty$-category and fix two objects $X,Y \in A$, along with two degree 0 closed morphisms $f,g \in \hom_A(X,Y)$ in the same cohomology class (i.e., $f$ and $g$ are homotopic). Then the objects $\cone(f)$ and $\cone(g)$ in $\Tw A$ are isomorphic in $H^0(\Tw A)$. 
\end{prop}

\begin{proof}
While it is possible to prove the proposition explicitly using $A_\infty$ relations, we invoke a different argument. First observe that the claim is true if $A$ is a dg-category. Our goal is to reduce the claim to the dg-category case.

Let us recall the Yoneda embedding $\cY : \Tw A \to \fun_{A_\infty}(\Tw A)^{\op},\chain_\kk)$. We will let $\cY(\Tw A)$ denote the full subcategory of $\fun_{A_\infty}( (\Tw A)^{\op}, \chain_\kk)$ spanned by $\cY(\ob \Tw A)$. 

Then the map $\Tw A \to \cY(\Tw A)$ is an equivalence of $A_\infty$-categories, meaning it admits an inverse up to natural equivalence. (See Corollary~A.9 of~\cite{lyubashenko-manzyuk-bimodules} and the references there.) In particular, it is a quasi-equivalence. By Proposition~13.19 of~\cite{bespalov-lyubashenko-manzyuk}, the induced functor of $H^0$ categories is triangulated. All we will need to know about the pretriangulated structures of the domain in codomain is that mapping cone sequences gives rise to examples of distinguished triangles (see Definition~13.18 of ibid., taking $n=1$):  In particular, by the triangulated category axioms, we find that in the 0th cohomology category $H^0(\cY(\Tw A))$, there exist isomorphisms
	\eqnn
	\cone(\cY(f)) \cong \cY(\cone(f)),
	\qquad
	\cone(\cY(g)) \cong \cY(\cone(g))
	\qquad \in H^0(\cY(\Tw A)).
	\eqnnd
(and likewise for $g$). Because $\Tw A \to \cY(\Tw A)$ is a quasi-equivalence, the hypothesis that $[f] = [g] \in H^0\hom_A(X,Y)$ implies that $[\cY(f)] = [\cY(g)]$. Appealing to the dg-category case, we conclude that
	\eqnn
	\cone(\cY(f)) \cong \cone(\cY(g)) \in H^0(\cY(\Tw \cA)). 
	\eqnnd
We conclude from the above two inline equations that
	\eqnn
	\cY(\cone(f)) \cong \cY(\cone(g)).
	\eqnnd
Because the Yoneda embedding is an equivalence onto its image, we have
	\eqnn
	\cone(f) \cong \cone(g) \in H^0(\Tw A).
	\eqnnd
\end{proof}

\begin{prop}
\label{prop. localizing cohomologous}
Fix a unital $A_\infty$-category $A$.
Suppose that $W_0 \subset W$ are collections of degree-zero, closed morphisms in $A$ such that $W_0$ and $W$ define the same collection of morphisms in the cohomology category $H^0 A$.
The functor
	\eqnn
	A[W_0^{-1}] \to
	A[W^{-1}]
	\eqnnd
is a quasi-equivalence. In particular, up to quasi-equivalence, a localization $A[W^{-1}]$ is determined by $H^0(W)$. 
\end{prop}

\begin{proof}[Proof of Proposition~\ref{prop. localizing cohomologous}.]
Let $\cone(W_0) \subset \Tw A$ denote the full subcategory whose objects are mapping cones of the morphisms $w_0 \in W_0$. We let $\cone(W) \subset \Tw A$ be the full subcategory spanned by mapping cones morphisms $w \in W$.

By Recollection~\ref{recollection. localizations of Aoo cats}, the localizations $A[W_0^{-1}]$ and $A[W^{-1}]$ are defined using the construction from~\cite{lyubashenko-manzyuk} of the quotient $A_\infty$-categories
	\eqnn
	D_0 := \sD(\Tw A | \cone(W_0)),
	\qquad
	D := \sD(\Tw A | \cone(W)).
	\eqnnd
One has maps $p_0: \Tw A \to D_0$ and $p: \Tw A \to D$. Moreover, by the naturality of quotients (Corollary~3.2 of~\cite{lyubashenko-ovsienko}), the identity map of $\Tw A$ induces an $A_\infty$-functor $D_0 \to D$, fitting into a commuting diagram
	\eqnn
	\xymatrix{
	\Tw A \ar[r]^{p_0} \ar[d]^= & D_0 \ar[d]_{j} \\
	\Tw A \ar[r]^{p} & D.
	}
	\eqnnd
Every functor above acts as the identity on objects.

We set another piece of notation. If $A$ is an $A_\infty$-category and $B \subset A$ is a full subcategory, for an arbitrary $A_\infty$-category $E$, let
	\eqnn
	\fun_{A_\infty}(A,E)_{\text{mod} B} \subset \fun_{A_\infty}(A,E)
	\eqnnd
denote the full subcategory of those functors $F$ for which $F$ is contractible along $B$ (Definition~\ref{defn. contractible}). Then the quotient category $\sD(A|B)$ satisfies the following property: For any $A_\infty$-category $E$, the restriction along $A \to \sD(A|B)$ 
	\eqnn
	\fun_{A_\infty}( \sD(A|B), E)
	\to
	\fun_{A_\infty}( A, E)_{\text{mod}B}
	\eqnnd
admits an inverse up to natural equivalence. (To see this, combine Theorem~1.3 of~\cite{lyubashenko-manzyuk}, which proves the universal property for the model $\sq(A|B)$ of the quotient, then invoke Proposition~7.4 of ibid., which states that there is a functor $\sq(A|B) \to \sD(A|B)$ admitting an inverse up to natural equivalence.) 

We now claim that the functor $j: D_0 \to D$ is a quasi-equivalence of $A_\infty$-categories. This will take a few paragraphs.

Combining the observations above, any $A_\infty$-category $E$ induces a commutative diagram of $A_\infty$-categories
	\eqn\label{eqn. universal property square for D and D0}
	\xymatrix{
	\fun_{A_\infty}( \Tw A, E)_{\text{mod}\cone(W_0)}
		& \fun_{A_\infty}( D_0, E)   \ar[l]_-{\simeq} \\
	\fun_{A_\infty}( \Tw A, E)_{\text{mod}\cone(W)} \ar[u]
		&	\fun_{A_\infty}( D,E) \ar[l]_-{\simeq} \ar[u]_{j^*}
	}
	\eqnd
where the horizontal arrows admit inverses up to natural equivalence. The righthand vertical arrow is the restriction along $j: D_0 \to D$, while the lefthand verticle arrow is the full inclusion induced by the observation that if a functor $F$ is contractible along all of $\cone(W)$, it is in particular contractible along $\cone(W_0)$. 

In fact, the lefthand vertical arrow in~\eqref{eqn. universal property square for D and D0} is not only an injection on objects, but a bijection on objects. To see this, we need only prove that if $F: \Tw A \to E$ is contractible along $\cone(\id_0)$, then it is contractible along $\cone(W)$. This is straightforward: By hypothesis of unit, any $w \in W$ is cohomologous to some $w_0 \in W_0$. Thus, by Proposition~\ref{prop. homotopic cones}, we have an isomorphism in $H^0(\Tw A)$ from $\cone(w_0)$ to $\cone(w)$. So let us choose such a representative homotopy-invertible map:
	\eqnn
	a^X: \cone(w_0) \to \cone(w)
	\eqnnd
in $\Tw A$. Then the for any other $w' \in W$ with cohomologous $w_0' \in W_0$, the definition of $A_\infty$-functor gives rise to a homotopy-commuting diagram of cochain complexes and chain maps
	\eqnn
	\xymatrix{
	\hom_{\Tw A}(\cone(w_0),\cone(w_0')) \ar[rr]_-{\simeq}^-{ (a^X)^\ast(a^Y)_\ast} \ar[d]^F
	&&	\hom_{\Tw A}(\cone(w),\cone(w'))  \ar[d]^F \\
	\hom_{E}(F\cone(e_0^X),F\cone(w_0') \ar[rr]_-{\simeq}^-{ (Fa^X)^\ast(Fa^Y)_\ast}
	&&	\hom_{E}(F\cone(w),F\cone(w')) .
	}
	\eqnnd
(There is ambiguity in the notation $(a^X)^\ast (a^Y)_\ast$, as one could mean either of
	\eqnn
	\bullet \mapsto m^2(m^2(a^Y,\bullet),a^X),
	\qquad
	\bullet \mapsto m^2(a^Y,m^2(\bullet,a^X)),
	\eqnnd
and of course these two operations are homotopic by the $A_\infty$-relations; so the diagram commutes up to homotopy regardless of which of the above one means.)  Because $a^X$ and $a^Y$ are homotopy invertible, so are their push- and pull- maps; so the horizontal arrows are homotopy equivalences of chain complexes. In particular, if the lefthand downward arrow is null-homotopic, so is the righthand downward arrow. This proves that the lefthand vertical arrow in~\eqref{eqn. universal property square for D and D0} is a bijection on objects.

In particular, the arrow is an isomorphism of $A_\infty$-categories. We conclude from the other equivalences in~\eqref{eqn. universal property square for D and D0} that the restriction map
	\eqnn
	j^*: 
	\fun_{A_\infty}(D,E) \to
	\fun_{A_\infty}(D_0,E) 
	\eqnnd
admits an inverse up to natural equivalence.

Setting $E = D_0$ and choosing a homotopy inverse $g$ to $j^*$, we see there exists a functor $g(\id_{D_0}) = F: D \to D_0$ for which $j^* F = F \circ j$ admits a natural equivalence to $\id_{D_0}$. On the other hand, the restriction $F^*$ must be homotopic to $g$ by uniqueness of inverses. We conclude that $F$ and $j$ are inverse functors up to homotopy. 

This proves our claim that $j$ is a quasi-equivalence (in fact, it is invertible up to natural equivalence).

We conclude the proof of Corollary~\ref{cor. id0 and id} as follows: $A[W_0^{-1}]$ is the full subcategory of $D_0$ spanned by $\ob A$. Likewise, $A[W^{-1}]$ is the full subcategory of $D$ spanned by $\ob A$. Because $j: D_0 \to D$ is the identity on objects, the result follows.
\end{proof}

\begin{remark}[Understanding hom complexes]
Fix a unital $A_\infty$-category $A$ and fix two objects $X,Y \in A$. 
Let us understand the morphism complex $\hom_{\Aploc}(X,Y)$ using Recollection~\ref{recollection. localizations of Aoo cats}. For $l \geq 2$, the $l$th associated graded of the length filtration (i.e., the quotient of the length $\leq l$ words by the length $\leq l-1$ words) is -- as a chain complex, and after a shift of $s^{-l+1}$ -- a direct sum
	\eqn\label{eqn. hom complex summands}
	\bigoplus_{e_1,\ldots,e_{l-1}} \hom_{\Tw A^+}(C_{l-1},Y) | \hom_{\Tw A^+}(C_{l-2},C_{l-1}) | \ldots | \hom_{\Tw A^+}(X, C_1)
	\eqnd
where each $e_i: W_i \to W_i$ is a unit in $A$ and $C_i = \cone(e_i) \in \Tw A^+$. (We set $W_0 = X$ and $W_l = Y$ for brevity of notation.)
\end{remark}

\begin{notation}
In contrast with our diligent use of $s$ in~\eqref{eqn. bar complex homs}, we will often drop the shifts when expressing a single associated graded component, as we have in~\eqref{eqn. hom complex summands}. 
\end{notation}

%
%

We will make use of the following corollary of Proposition~\ref{prop. localizing cohomologous} in the proof of Lemma~\ref{lemma. right inverse}:

\begin{corollary}
\label{cor. id0 and id}
Fix a unital $A_\infty$-category $A$.
Using the axiom of choice, choose a unit $e_0$ for every object $X \in \ob A$, and let $\id_0$ denote the collection $\{e_0\}_{X \in \ob A}$. The functor
	\eqnn
	\Aploco \to \Aploc
	\eqnnd
is a quasi-equivalence.
\end{corollary}

\begin{proof}[Proof of Lemma~\ref{lemma. right inverse}.]
The inclusion is a bijection on objects. So it suffices to show that the maps on hom complexes are quasi-isomorphisms. To this end, 
fix two objects $X,Y \in \ob A$. 

{\em When $X \neq Y$.} Let us study the length filtration of $\hom_{\Aploc}(X,Y)$ (Recollection~\ref{recollection. localizations of Aoo cats}).  The length $l=1$ part is $\hom_A(X,Y)$ because $X \neq Y$. We now claim the $l$th associated gradeds for $l \geq 2$ are acyclic. Consider the direct summand of~\eqref{eqn. hom complex summands} corresponding to a tuple $(e_1,\ldots,e_{l-1})$. Here, we treat this summand as a direct summand of the $l$th associated graded quotient -- in particular, its differential involves only $m^1_{\Tw A^+}$ terms.

Because $X \neq Y$, there exists at least one $i$ for which $W_i \neq W_{i+1}$.  
\begin{itemize}
	\item If we can choose such an $i$ to satisfy $l-2 \geq i \geq 1$, we have that
	\eqnn
	\hom_{\Tw A}(C_i, C_{i+1}) = 
	\hom_{\Tw A^+}(C_i, C_{i+1}).
	\eqnnd
On the other hand, $\hom_{\Tw A}(C_i, C_{i+1})$ is a mapping cone for the morphism $Z$ identified in~\eqref{eqn. cone cone differential}. (Using the notation there, $e = e_i$ and $e' = e_{i+1}$.) In particular, the direct summand is a mapping cone for the chain map
	\eqnn
	\id^{| (l-i-1)} | Z | \id^{| i}
	:=
	\id | \ldots | \id | Z | \id | \ldots | \id
	\eqnnd
which is homotopy-invertible because $Z$ is; so the direct summand is acyclic. 
\item If $i = 0$ or $i=l-1$, then the direct summand is a mapping cone for the homotopy-invertible map $\id^{|( l-1)} | m^2(e_1,-)$ or $m^2(-,e_{l-1}) | \id^{|( l-1)}$, so is acyclic.
\end{itemize} 
All told, the $l$th associated graded is a direct sum of acyclic complexes,  hence acyclic. This shows that the map $\hom_A(X,Y) \to \hom_{\Aploc}(X,Y)$ is a quasi-isomorphism when $X \neq Y$.

{\em When $X = Y$.} 
Choose a unit $e_0^X$ for every object $X \in \ob A$ and
consider the chain maps
	\eqnn
	\hom_A(X,X) \to \hom_{\Aploco}(X,X)
	\to \hom_{\Aploc}(X,X).
	\eqnnd
Corollary~\ref{cor. id0 and id} tells us the second map is a quasi-isomorphism, so it suffices to prove that the first arrow is a quasi-isomorphism. 

Let $\TT \subset \hom_{\Aploco}(X,X)$ be the subcomplex whose length 1 filtration consists of all of $\hom_{A^+}(X,X)$, and whose length $l \geq 2$ compoent consists of the summand in~\eqref{eqn. hom complex summands} for which $e_1 = \ldots = e_{l-1} = e_0^X \in \id_0$. 
Then the quotient $\hom_{\Aploco}(X,X) / \TT$, as a graded $\kk$-module, is isomorphic to a direct sum of those summands in~\eqref{eqn. hom complex summands} for which $W_i \neq W_{i+1}$ for some $i$ (and -- though this won't matter -- for which each $e_i$ is in $\id_0$). By a similar argument as in the case $X \neq Y$, we conclude that the length filtration of $\hom_{\Aploco}(X,X) / \TT$ has acyclic associated gradeds. In particular, we conclude that $\hom_{\Aploco}(X,X) / \TT$ is acylic, hence we are left to prove that the inclusion $\hom_A(X,X) \to \TT$ is a quasi-isomorphism.

For this, consider the subcomplex $\SS \subset \TT$ where the length $l=1$ component of $\SS$ is isomorphic to $\hom_A(X,X)$, and the length $l\geq 2$ component of $\SS$ is
	\eqn\label{eqn. S length summand}
	\hom_{\Tw A}(\cone(e_{0}^X),X) | \hom_{\Tw A^+}(\cone(e_{0}^X),\cone(e_{0}^X)) | \ldots | \hom_{\Tw A^+}(X,\cone(e_{0}^X)).
	\eqnd
Note the (important!) subscript $\Tw A$ in the first bar factor, in contrast to the $\Tw A^+$ subscripts in all other factors. We will prove the following two claims:

\enum[(I)]
	\item\label{item. t homotopy} There exists a chain map $t: \TT \to \SS$ so that the composition $\TT \xrightarrow{t} \SS \xrightarrow{\subset} \TT$ is homotopic to the identity.
	\item\label{item. t and homA} For the same map $t$ as above, the composition
		$\hom_A(X,X) \xrightarrow{\subset} \TT \xrightarrow{t} \SS$ has image contained in $\hom_A(X,X)$. Moreover, this composition -- treated as a self-chain-map of $\hom_A(X,X)$ -- is homotopic to the identity map $\id_{\hom_A(X,X)}$.
\enumd
We first prove Claim~\eqref{item. t homotopy}. 

Define a degree -1 operator $H$ on $\TT$ acting on generators by
	\begin{align}
	\bx_l \vert \ldots \vert \bx_1
	& \mapsto
	(a \tensor 1_X) \vert \bx_l \vert \ldots \vert \bx_1 \nonumber \\
	&\in \hom_{\Tw A^+}(\cone(e_{l-1}),X) | \hom_{\Tw A^+}(\cone(e_{l-1}),\cone(e_{l-1})) | \ldots |  \hom_{\Tw A^+}(X,\cone(e_1))\nonumber.
	\end{align}
Let us explain the notation.
\begin{itemize}
	\item Every $e_1, e_2, \ldots, e_{l-1}$ is equal to $e^X_0$, but we have written the bar notation above to be explicit about the length of $H(\bx_l | \ldots | \bx_1)$. In particular, note that $H$ increases length filtration by 1. 
	\item $a \tensor 1_X$ follows the notation from~\eqref{eqn. abcd sum tensor}. For brevity of notation, we are identifying $\hom(-,X)$ with a subcomplex of $\hom(-,\cone(e_0^X))$, setting $q=y=0$.
	\item $1_X \in \hom_{A^+}(X,X)$ is the augmentation unit of $X$ (Recollection~\ref{recollection. augmentation}).
	\item In the image of $H$, the letter $\bx_l$ is now an element of $\hom_{\Tw A^+}(\cone(e_{l-1}),\cone(e_{l-1}))$ rather than an element of $\hom_{\Tw A^+}(\cone(e_{l-1}),X)$; as before, the latter complex is the subcomplex for which $q=y=0$.
\end{itemize}
Then $H$ is a homotopy between the maps 
	\eqnn
	\id_{\TT}
	\qquad
	\text{and}
	\qquad
	\bx_l \vert \ldots \vert \bx_1
	\mapsto
	\pm(b \tensor e_{0}^X) \vert \bx_l \vert \ldots \vert \bx_1
	\eqnnd
where $b$ is as in~\eqref{eqn. abcd sum tensor}. (We point out here the importance of working in $\hom_{\Aploco}(X,X)$, and hence the utility of Corollary~\ref{cor. id0 and id}. To construct a homotopy as above inside $\hom_{\Aploc}(X,X)$ involves much more algebra.)

Next, choose a degree -1 endomorphism $\alpha \in \hom_A(X,X)$ realizing a homotopy from $e_0^X$ to $m^2_A(e_0^X,e_0^X)$. Define the degree -1 operator $G: \TT \to \TT$ by
	\eqnn
	G: 
	\bx_l \vert \ldots \vert \bx_1
	\mapsto
	(b \tensor \alpha) \vert \bx_l \vert \ldots \vert \bx_1  .
	\eqnnd
Then $G$ is a homotopy between the maps sending $
	\bx_l \vert \ldots \vert \bx_1$ to
	\eqnn
	\pm(b \tensor e_0^X) \vert \bx_l \vert \ldots \vert \bx_1
	\qquad{\text{and}}\qquad
	\pm(b \tensor m^2_A(e_0^X,e_0^X)) \vert \bx_l \vert \ldots \vert \bx_1.
	\eqnnd
(To see this, it is important to note that $\bx_l$ only has $a$ and $b$ tensor factors, and $ba = bb = 0.$)
Finally consider the degree -1 operator $K: \TT \to \TT$	\eqnn
	K: 
	\bx_l \vert \ldots \vert \bx_1
	\mapsto
	(a \tensor e_0^X) \vert \bx_l \vert \ldots \vert \bx_1  .
	\eqnnd
Then $K$ is a homotopy between  the maps sending $\bx_l \vert \ldots \vert \bx_1$ to
	\eqnn
	\pm(b \tensor m^2(e_0^X,e_0^X)) \vert \bx_1 \vert \ldots \vert \bx_1
	\qquad{\text{and}}\qquad
	\sum_{k \geq 2} \pm m^k_{\Tw A^+}(a \tensor e_0^X, x_{l},\ldots,x_{l-k+2}) \vert x_{l-k+1} \vert \ldots  \vert x_1.	
	\eqnnd
We call this last map $t'$. Then the degree -1 maps $H, G, K$ combine to define a homotopy
	\eqn\label{eqn. homotopy to t'}
	\id_{\TT} \sim t'.
	\eqnd
We now note that $t'$ has image contained in $\SS$. We accordingly define $t$ as having the same effect as $t'$, but as a map with codomain $\SS$:
	\eqnn
	t: \TT \to \SS,
	\qquad
	\bx_l \vert \ldots \vert \bx_1
	\mapsto
	\sum_{k \geq 2} \pm m^k_{\Tw A^+}(a \tensor e_{0}^X, x_{l},\ldots,x_{l-k+2}) \vert x_{l-k+1} \vert \ldots  \vert x_1.	
	\eqnnd
By definition, $t'$ is the composition $\TT \xrightarrow{t} \SS \xrightarrow{\subset} \TT$, so Claim~\eqref{item. t homotopy} follows from~\eqref{eqn. homotopy to t'}.

We now prove Claim~\eqref{item. t and homA}. We see that for an element $x \in \hom_A(X,X)$, 
	\eqnn
	t(x)
	= m^2_A(e_0^X,x).
	\eqnnd
So indeed $t(\hom_A(X,X)) \subset \hom_A(X,X)$. On the other hand, by definition of unit, the map $m^2_A(e_0^X,-)$ is chain-homotopic to the identity map. We have proven the claim.

Recall we were left to prove that $\hom_A(X,X) \to \TT$ is a quasi-isomorphism. 
By observing (using the same old trick as above) that the length filtration on $\SS$ has contractible associated gradeds for $l \geq 2$, we see that the inclusion $\hom_A(X,X) \to \SS$ is a quasi-isomorphism. And our two claims give rise to the following homotopy-commuting diagram of cochain complexes:
	\eqnn
	\xymatrix{
	\hom_A(X,X) \ar[dr]_-{\sim} \ar[r]^-{\subset} 
		& \TT \ar[dr]^-{\id_{\TT}} \ar[d]_{t} 
	\\
		& \SS \ar[r]^-{\subset}
		& \TT
	}
	\eqnnd
It follows that the inclusion $\hom_A(X,X) \to \TT$ is a quasi-isomorphism, as desired.
\end{proof}

\subsection{Proof of Theorem~\ref{theorem. main theorem}}

\begin{lemma}
\label{lemma. j has left inverse}
The functor $\Ainftystr[\eqs^{-1}] \to \Ainfty[\eqs^{-1}]$ induced by $j$ admits a left inverse up to natural equivalence.
\end{lemma}

The following proof is due to Canonaco-Ornaghi-Stellari, who relayed it to me in private communication. We thank them for generously sharing their ideas and allowing us to include their proof, greatly simplifying the present exposition! The old, computational proof has been banished to the appendix, in case the computations and observations there are of use to others.

\begin{proof}
There is an endofunctor $\tau': \Ainftystr \to \Ainftystr$ sending $A$ to $A[\id_A^{-1}]$, and this is equipped with a natural transformation $\eta'$ via the obvious inclusions
	\eqnn
	\eta'_A: A \to A[\id_A^{-1}].
	\eqnnd
The same strategy as the proof of Lemma~\ref{lemma. right inverse} -- noting that the length 1 piece of $\hom_{A[\id_A^{-1}]}(X,Y)$ is equal to $\hom_A(X,Y)$, and that all length $l \geq 2$ filtered pieces are direct sums of acyclic complexes -- shows that $\eta'_A$ is a quasi-equivalence for all $A$. In particular, $\tau'$ respects quasi-equivalences and induces an endofunctor of $\Ainftystr[\eqs^{-1}]$. 

Because $\eta_A'$ factors the quasi-equivalence $A \to A^+[\id_A^{-1}]$ in Lemma~\ref{lemma. right inverse}, we conclude that the map
	\eqn\label{eqn. obvious equivalence}
	A[\id_A^{-1}] \to
	A^+[\id_A^{-1}]
	\eqnd
is also a quasi-equivalence.

On the other hand, we have functors $A \to A^+ \to A$ (the first map is a non-unital inclusion; the second, which identifies the strict and augmentation units, is strictly unital). These functors respect $\id_A$ as sets, so we have induced maps of localizations
	\eqnn
	A[\id_A^{-1}]
	\to
	A^+[\id_A^{-1}]
	\xrightarrow{\eta''_A}
	A[\id_A^{-1}].
	\eqnnd
Because the composition $A \to A^+ \to A$ is the identity functor of $A$, the above composition is the identify functor of $A[\id_A^{-1}]$ (hence a quasi-equivalence). The first map is a quasi-equivalence as already noted in~\eqref{eqn. obvious equivalence}. It follows that the second map $\eta''_A$ is a quasi-equivalence. The map $\eta''_A$ is natural in $A$, so we witness natural transformations
	\eqnn
	\xymatrix{
	\id_{\Ainftystr} \ar[r]^{\eta'}
		& \tau'
		& \tau \circ j \ar[l]_{\eta''}
	}
	\eqnnd
where $\eta'_A$ and $\eta''_A$ are quasi-equivalences for all $A$.
By Proposition~\ref{prop. good natural transformations induce homotopies}, the natural transformations induced by $\eta'$ and $\eta''$  are natural equivalences upon localizing along $\eqs$.

Thus, the functor induced by $\tau \circ j$ is naturally equivalent to the identity functor, and this exhibits the functor induced by $\tau$ as a left inverse to the functor induced by $j$.
\end{proof}

\begin{proof}[Proof of Theorem~\ref{theorem. main theorem}.]
The functor induced by $j$ admits a right inverse (Remark~\ref{remark. right inverse}) and a left inverse (Lemma~\ref{lemma. j has left inverse}) up to homotopy. 
\end{proof}

\section{Internal homs and mapping spaces} 
\label{section. interhal homs and maps}
The main results of this section -- Theorems~\ref{theorem. internal homs in dgcat},
\ref{theorem. maps in dgcat},
\ref{theorem. internal homs in Aoocat}, and
\ref{theorem. mapping spaces in Aoocat} -- rely on many structural results in the study of $A_\infty$-categories and in the study of $\infty$-categories. 
While basic knowledge in both areas is assumed, we include  proofs of some well-known results, together with references for filling in the details, to invite more readers into the fold.

The bulk of this section is occupied with the task of comparing structures arising from the theory of $\infty$-categories (which are completely formalism-driven) with known dg- and model-categorical constructions (which are formula- and point-set-driven). For example, we identify the monoidal structure in $\Ainftycat$ (Proposition~\ref{prop. monoidal structure is derived dg tensor}), thereby characterizing the internal $\underline{\hom}$ objects, and we identify the right adjoint to~\eqref{eqn. free dg category} (Proposition~\ref{prop. dg nerve is right adjoint}) thereby computing mapping spaces.

\subsection{The $\infty$-category of chain complexes}

\begin{prop}
\label{prop. cofibrant localization is full localization}
Let $C$ be a model category with functorial cofibrant resolutions -- i.e., a functor $Q: C \to C$ whose images are all cofibrant, equipped with a natural transformation $Q \to \id$ consisting of weak equivalences. We let $C^\circ \subset C$ be the full subcategory of cofibrant objects. Then the map of $\infty$-categorical localizations (along weak equivalences)
	\eqnn
	C^\circ[W^{-1}]
	\to
	C[W^{-1}]
	\eqnnd
is an equivalence of $\infty$-categories.
\end{prop}

\begin{proof}
Let $i: C^\circ \to C$ denote the inclusion. Then $Q \circ i$ enjoys a natural transformation to $\id_{C^\circ}$ consisting of weak equivalences. Likewise, $i \circ Q$ enjoys a natural transformation of weak equivalences to $\id_C$. Thus, upon passage to the localizations, the induced functors of $Q$ and $i$ are mutually inverse functors of $\infty$-categories.
\end{proof}

The following is a special case of Proposition~4.1.3.4 of~\cite{higher-algebra}, and Example~4.1.3.6 of ibid.

\begin{prop}[Localization is symmetric monoidal]
\label{prop. localization is symmetric monoidal}
Let $C$ be a monoidal category, and $W \subset C$ a class of morphisms containing identities. Further assume that $W \tensor W \subset W$. Then there exists a symmetric monoidal structure on the $\infty$-category $C[W^{-1}]$ for which
\enum[(i)]
\item the localization map of $\infty$-categories $C \to C[W^{-1}]$ has a natural promotion to a symmetric monoidal functor, and 
\item\label{item. universal property of symm monoidal localization} $C[W^{-1}]$ is universal for symmetric monoidal functors out of $C$ sending $W$ to equivalences. \enumd
Moreover, the assignment $(C,W) \mapsto C[W^{-1}]$ is itself symmetric monoidal. (In particular, it is functorial with respect to the direct product of pairs $(C,W)$ and maps $C \to C'$ sending $W$ to $W'$.)
\end{prop}

\begin{corollary}
\label{corollary. symmetric monoidal model categories}
Let $C$ be a symmetric monoidal model category, meaning the monoidal unit is cofibrant, the monoidal structure is closed, and $\tensor: C \times C \to C$ is a (left) Quillen bifunctor. Then the localization map $C^\circ \to C^\circ[W^{-1}]$  (of cofibrant objects along weak equivalences) has a natural promotion to a symmetric monoidal functor. 
\end{corollary}

\begin{remark}
To be absolutely innocent of abuse, one should technically write $N(C^\circ) \to C^\circ[W^{-1}]$ for the localization map -- as $C^\circ$ is a category, and $N(C^\circ)$ is the nerve (thereby rendering $C^\circ$ an $\infty$-category). We note that if $K$ is a symmetric monoidal category, then $N(K)$ is a symmetric monoidal $\infty$-category (Example 2.1.2.21 of~\cite{higher-algebra}).
\end{remark}

\begin{example}
\label{example. symm mon maps of chain}
Let $\chain$ denote the 1-category of (possibly unbounded) chain complexes over our commutative ring $\kk$.
We endow $\chain$ with the projective model structure.\footnote{See Hovey's book on model categories~\cite{hovey-model-categories}, Proposition~4.2.13. Though not every projective complex is cofibrant, it is true that every cofibrant object is a complex of projectives.} It is well-known that $\chain$ is a symmetric monoidal model category with functorial cofibrant replacements (because $\chain$ is combinatorial, for example). Then:
	\enum[(a)]
	\item Letting $\chain^\circ$ denote the full subcategory of fibrant and cofibrant objects, the induced functor
		\eqnn
		\chain^\circ[\quis^{-1}]
		\to\chain[\quis^{-1}]
		\eqnnd
	(where we have localized along quasi-isomorphisms) is an equivalence of $\infty$-categories (Proposition~\ref{prop. cofibrant localization is full localization}).
	\item $\chain$ is symmetric monoidal under the usual tensor product $\tensor = \tensor_{\kk}$, and the tensor product of two cofibrant chain complexes is cofibrant, while preserving quasi-isomorphisms of cofibrant chain complexes. By Proposition~\ref{prop. localization is symmetric monoidal}, there is an induced symmetric monoidal structure on the localization, and the functor
		\eqnn
		\chain^\circ \to \chain^\circ[\quis^{-1}]
		\eqnnd
	is a symmetric monoidal functor of symmetric monoidal $\infty$-categories.
	\enumd
\end{example}

In fact, there is a larger class of chain complexes than the cofibrant chain complexes that is useful to pick out: The homotopically projective complexes (see below). We do not know whether there is a model structure on (possibly unbounded) chain complexes for which the homotopically projective objects are precisely the cofibrant objects.

\begin{defn}\label{defn. homotopically projective}
Recall that a chain complex of $\kk$-modules $M$ is {\em homotopically projective} if $H^*(N)\cong 0 \implies H^*\hom(M,N) \cong 0$. $M$ is called {\em homotopically flat} if $H^*(N) \cong 0 \implies H^*(M \tensor_{\kk} N) \cong 0$.
\end{defn}

\begin{remark}
\label{remark. good properties of h projectives}
By the tensor-hom adjunction, it is clear that the tensor product of homotopically projective complexes is homotopically projective. 
It is a result of Spaltenstein that every homotopically projective complex is homotopically flat.\footnote{Proposition 5.8 of~\cite{spaltenstein-resolutions}. Note that Spaltenstein uses the term $K$-flat and $K$-projective, while we follow the terminology of Drinfeld~\cite{drinfeld-dg-quotients}.} Moreover, every cofibrant chain complex (in the projective model structure) is homotopically projective (Lemma~2.3.8 of~\cite{hovey-model-categories}).
\end{remark}

\begin{prop}[Whitehead's Theorem for homotopically projective objects]
\label{prop. whitehead for projectives}
Let $X$ and $Y$ be homotopically projective chain complexes over $\kk$. If $f: X \to Y$ is a quasi-isomorphism, then $f$ admits an inverse chain map up to homotopy.
\end{prop}

\begin{proof}
Because $Y$ is homotopically projective, there exists a homotopy right inverse $g: Y \to X$ -- so $fg$ is chain homotopic to $\id_Y$ (Proposition~1.4 of~\cite{spaltenstein-resolutions}). It follows that $g$ is a quasi-isomorphism.  Because $X$ is homotopically projective, $g$ admits a homotopy right inverse. By uniqueness of 2-sided inverses, we see that $f$ and $g$ are mutually homotopy-inverse chain maps.

\end{proof}

\begin{prop}
\label{prop. homproj and homcof are equivalent}
Let $\chain^{\homproj} \subset \chain$ denote the full subcategory consisting of homotopically projective chain complexes. Then the symmetric monoidal inclusion
	$
	\chain^{\circ} \to \chain^{\homproj}
	$
induces a symmetric monoidal equivalences of $\infty$-categories
	\eqnn
	\chain^{\circ}[\quis^{-1}] \to \chain^{\homproj}[\quis^{-1}].
	\eqnnd
\end{prop}

\begin{proof}
Localize the functors $\chain^{\circ} \subset \chain^{\homproj} \subset \chain \xrightarrow{Q} \chain^{\circ}$ along $\quis$ and use Proposition~\ref{prop. cofibrant localization is full localization}. Apply Proposition~\ref{prop. localization is symmetric monoidal} for the symmetric monoidal enhancement.
\end{proof}

\begin{defn}
\label{defn. oo cat of chain}
By the {\em $\infty$-category of chain complexes} over $\kk$, we mean any of the equivalent $\infty$-categorical localizations
	\eqnn
	\chain^\circ[\quis^{-1}]\simeq
	\chain^{\homproj}[\quis^{-1}]\simeq
	\chain[\quis^{-1}]
	\eqnnd
where $\chain$ is the 1-category of (possibly unbounded) $\kk$-linear chain complexes, and $\quis$ is the collection of quasi-isomorphisms. 

By the {\em symmetric monoidal} $\infty$-category of chain complexes over $\kk$, we mean any of the above endowed with the symmetric monoidal structure induced by Proposition~\ref{prop. localization is symmetric monoidal}. (We caution we do not make use of the proposition for $\chain[\quis^{-1}]$. Instead, we must pushforward/pullback the symmetric monoidal structure to $\chain[\quis^{-1}]$ along any of the above equivalences.)
\end{defn}

\subsection{Categories of dg-categories}

\begin{notation}
\label{notation. dgcat versions}
Fixing our base ring $\kk$, we let $\dgcatt$ denote the usual category of $\kk$-linear dg-categories. 
We have full subcategories 
	\eqnn
	\dgcatt^\circ \subset \dgcatt_{\homcof} \subset \dgcatt_{\homproj} \subset \dgcatt
	\eqnnd
where 
\begin{itemize}
	\item $\dgcatt^{\circ}$ consists of the cofibrant dg-categories with respect to the Tabuada model strcuture.
	\item $\dgcatt_{\homcof}$ consists of dg-categories whose morphism complexes are cofibrant chain complexes in the projective model structure for $\chain$. (All cofibrant dg-categories have cofibrant morphism complexes by Proposition~2.3(3) of~\cite{toen-homotopy-theory-of-dg-cats}.) Finally,
	\item $\dgcatt_{\homproj}$ consists of dg-categories whose morphism complexes are homotopically projective. (All cofibrant chain complexes are homotopically projective by Remark~\ref{remark. good properties of h projectives}.)
\end{itemize}
\end{notation}

\begin{remark}
\label{remark. tensor of homcof are homcof}
$\dgcatt_{\homcof}$ and $\dgcatt_{\homproj}$ have the pleasant property that if $f$ and $g$ are functors in these categories and are quasi-equivalences, then $g \tensor f$ is also a quasi-equivalence. (This is because homotopically projective complexes are homotopically flat.) Moreover, the class of cofibrant chain complexes is closed under $\tensor$, as is the class of homotopically projective chain complexes -- so $\dgcatt_{\homcof}$ and $\dgcatt_{\homproj}$ are a symmetric monoidal subcategory of $\dgcatt$. (This last property is not enjoyed  by $\dgcatt^{\circ}$.) 
\end{remark}

\subsection{Enrichments}
\begin{notation}[$\algcat$]
For a monoidal $\infty$-category $\cV$, one has an $\infty$-category 
	\eqnn
	\algcat(\cV)
	\eqnnd
defined by Gepner-Haugseng. Informally, $\algcat(\cV)$ is an $\infty$-category where an object is the data of 
\enum[(i)]
	\item  a space $X$ of objects, 
	\item a map from $X \times X$ to the space of objects of $\cV$ -- in particular, for all $x,y \in X$, an object $\hom(x,y) \in \cV$ -- and 
	\item coherence data for a composition operation.\footnote{See the discussion after Definition~2.3.4 for a non-infinity-categorical example. A near-geodesic path of reading for the $\infty$-categorical case is: from Definition~4.3.1, see Corollary 4.2.8, then Definition~4.2.4. See also (the paragraph preceding) Definition~2.4.5, and Remark~2.4.7, of ibid.} 
\enumd 
A morphism in $\algcat(\cV)$ is a map of object spaces along with data specifying coherent compatibilities of the  composition operations. 
We refer to Definition~4.3.1 of~\cite{gepner-haugseng} for details.
	\eqnn
	\algcat(\cV)_{\Set}
	\subset
	\algcat(\cV)
	\eqnnd
is the full subcategory of $\algcat(\cV)$ whose spaces of objects are discrete (see Theorem 5.3.17 of~\cite{gepner-haugseng}). That is, an object of $\algcat(\cV)_{\Set}$ has a space of objects whose connected components are all contractible.

The assignment $\cV \mapsto \algcat(\cV)_{\Set} \subset \algcat(\cV)$ is functorial with respect to monoidal functors in the $\cV$ variable -- in fact, with respect to {\em lax} monoidal functors in the $\cV$ variable (Lemma 4.3.9 of~\cite{gepner-haugseng}).
\end{notation}

\begin{example}[dg-categories]
\label{example. usual enrichments of 1-categories}
If $\cV$ is (the nerve of) an ordinary category, then every object of $\algcat(\cV)_{\Set}$ is isomorphic to a $\cV$-enriched category in the usual sense. Importantly, the only equivalences in $\algcat(\cV)_{\Set}$ are enriched functors that induce a bijection on the set of objects, and an isomorphism on the morphism objects. 

As a sub-example, if $\cV=\chain$ is the usual category of chain complexes over $\kk$, $\algcat(\chain)_{\Set}$ is equivalent as an $\infty$-category to the usual category of dg-categories over $\kk$. A morphism in $\algcat(\chain)_{\Set}$ is an equivalence if and only if it is an isomorphism of dg-categories in the usual sense. In particular, quasi-equivalences are not invertible in $\algcat(\chain)_{\Set}$.
\end{example}

\begin{remark}[From $\algcat$ to enriched $\infty$-categories]
\label{remark. defn of enriched oo cats}
Let $\cV$ be a monoidal $\infty$-category.
In $\algcat(\cV)$, one can give sensible definitions of what it means for a morphism to be essentially surjective, or to be fully faithful (Section~5.3 of ibid). As already noted, essentially surjective and fully faithful functors are not typically invertible in $\algcat(\cV)$.
In this sense, $\algcat(\cV)$ is a precursor to the $\infty$-category of $\cV$-enriched $\infty$-categories. 

When $\cV$ is presentably monoidal (Definition~3.1.24 of~\cite{gepner-haugseng}) the $\infty$-category of $\cV$-enriched $\infty$-categories is obtained by localizing $\algcat(\cV)$ along the essentially surjective and fully faithful functors:\footnote{Strictly speaking, this is only true when $\cV$ is presentably monoidal -- see Corollary~5.6.4 of~\cite{gepner-haugseng}.}
	\eqn\label{eqn. two presentations for enriched cats}
	\inftyCat^{\cV} := \algcat(\cV)[\FFES^{-1}]
	\simeq  \algcat(\cV)_{\Set}[\FFES^{-1}].\footnote{In this work, we have presented this as a definition. In~\cite{gepner-haugseng}, the $\infty$-category of $\cV$-enriched $\infty$-categories is instead defined as a full subcategory of $\algcat(\cV)$ (thereby simplifying computations of mapping spaces). When $\cV$ is presentably monoidal, Proposition 5.4.4, Proposition~5.4.2, and Corollary~5.6.3 of ibid. show that the two characterizations are equivalent.}
	\eqnd
Note the equivalence: It is a theorem that localizing $\algcat(\cV)_{\Set}$ by the fully faithful and essentially surjective functors recovers the localization of $\algcat(\cV)$ along the fully faithful and essentially surjective functors (Theorem 5.3.17 of~ibid.).
\end{remark}

\begin{example}
\label{example. infinity categories as enriched}
Let $\spaces$ denote the $\infty$-category of spaces, for example modeled as the $\infty$-category of Kan complexes. We take direct product to be the symmetric monoidal structure. Then $\inftyCat^{\spaces}$ is equivalent to the $\infty$-category of $\infty$-categories.  See for example Section~4.4 of~\cite{gepner-haugseng} and Theorem 5.4.6 of ibid.
\end{example}

\subsection{Equivalent models of $\dgcat$ and $\Ainftycat$}
\label{section. equivalent dg cat models}
We have three equivalent models for the $\infty$-category of chain complexes (Definition~\ref{defn. oo cat of chain}). Accordingly:

\begin{notation}
Fix a commutative ring $\kk$. We define
	\eqnn
	\dgcat
	\eqnnd
to be any of the equivalent $\infty$-categories
	\begin{itemize}
	\item $\algcat(\chain^\circ[\quis^{-1}])_{\Set}[\FFES^{-1}]$
	\item $\algcat(\chain^{\homproj}[\quis^{-1}])_{\Set}[\FFES^{-1}]$, or
	\item $\algcat(\chain[\quis^{-1}])_{\Set}[\FFES^{-1}]$.
	\end{itemize}
We refer to $\dgcat$ as the $\infty$-category of dg-categories. 
\end{notation}

\begin{remark}
If $\cV$ is the $\infty$-category of chain complexes (Definition~\ref{defn. oo cat of chain}), $\cV$ is presentably symmetric monoidal. 
So by Remark~\ref{remark. defn of enriched oo cats}, $\dgcat$ is by definition the $\infty$-category of $\infty$-categories enriched over (the $\infty$-category of) chain complexes.   
\end{remark}

\begin{remark}
We will often prefer to model $\dgcat$ using $\chain^\circ$ and $	\chain^{\homproj}
	$
for the reason that the symmetric monoidal structure is directly induced by localization, and does not need to be defined by pulling back along equivalences of $\infty$-categories.
\end{remark}

We will now construct\footnote{A more general construction is given in Definition~5.1 of~\cite{haugseng-rectification-enriched}. For the reader's benefit, we note a notational subtley of ibid. : In the paragraph before Definition~3.2  of ibid, the author makes clear that ${\bf V}[W^{-1}]$ refers to ${\bf V}^\circ[W^{-1}]$ -- the localization of the subcategory of cofibrant objects.} a natural comparison map
	\eqn\label{eqn. comparison map}
	\dgcatt[\eqs^{-1}] \to \dgcat.
	\eqnd
For brevity we will use the model 
	$\dgcat = 
	\algcat(\chain^\circ[\quis^{-1}])_{\Set}[\FFES^{-1}]
	$
though the model using
	$
	\chain^{\homproj}
	$
works equally well.
	
\begin{construction}[The map \eqref{eqn. comparison map}]
By Example~\ref{example. symm mon maps of chain},
one has symmetric monoidal functors of $\infty$-categories
	\eqn\label{eqn. chain complexes and functors}
	\chain^\circ[\quis^{-1}] \leftarrow \chain^\circ \to \chain
	\eqnd
where we treat the ordinary categories $\chain$ and $\chain^\circ$ as $\infty$-categories (by, say, taking their nerves). This induces the top row in the following diagram of $\infty$-categories:
	\eqn\label{eqn. big diagram of dg cats}
	\xymatrix{
	\algcat(\chain^\circ[\quis^{-1}])_{\Set}   
	& \algcat(\chain^\circ)_{\Set} \ar[l]\ar[r] 
	& \algcat(\chain)_{\Set} 
	\\
		& \dgcatt_{\homcof} \ar[r] \ar[u]_{\simeq}
		& \dgcatt \ar[u]_{\simeq} 
	\\
		& \dgcatt^{\circ} \ar[u] \ar[ur]
	}.
	\eqnd
Because $\chain$ and $\chain^\circ$ are ordinary categories, the topmost vertical arrows are equivalences (Example~\ref{example. usual enrichments of 1-categories}). 
Because cofibrant dg-categories have cofibrant mapping complexes,\footnote{Proposition~2.3(3) of~\cite{toen-homotopy-theory-of-dg-cats}. A more general result for model categories that are not $\chain$ is invoked in Corollary 3.15 of~\cite{haugseng-rectification-enriched}.} $\dgcatt^\circ$ is a (full) subcategory of $\dgcatt_{\homcof}$. Composition of the leftmost arrows in~\eqref{eqn. big diagram of dg cats} thus induce a map
	\eqnn
	\dgcatt^\circ \to 
	\algcat(\chain^\circ[\quis^{-1}])_{\Set}.
	\eqnnd
Further, if a functor  of dg-categories is a quasi-equivalence, it is essentially surjective and fully faithful as a morphism in $\algcat(\chain^\circ[\quis^{-1}])_{\Set}$. Thus, \eqref{eqn. big diagram of dg cats} passes to localizations:
	\eqn\label{eqn. big diagram of oo cats of dg cats}
	\xymatrix{
	\dgcat   
	& \algcat(\chain^\circ)_{\Set}[\FFES^{-1}]  \ar[l]_-{\tensor} \ar[r]_{\simeq} 
	& \algcat(\chain)_{\Set}[\FFES^{-1}]  
	\\
		& \dgcatt_{\homcof}[\eqs^{-1}]  \ar[r]_{\simeq} \ar[u]_{\simeq}^{\tensor}
		& \dgcatt[\eqs^{-1}]  \ar[u]_{\simeq} 
	\\
		& \dgcatt^{\circ}[\eqs^{-1}]  \ar[u]_{\simeq} \ar[ur]_{\simeq}
	}
	\eqnd
(To justify that the topmost vertical arrows are still equivalences, we note that the $\FFES$ maps in $\algcat(\chain)_{\Set}$ are quasi-equivalences.) 
Any of the obvious composition of morphisms in the above diagram determines (up to homotopy) the desired map~\eqref{eqn. comparison map}.
\end{construction}

\begin{remark}
There are more equivalences in~\eqref{eqn. big diagram of oo cats of dg cats} than in~\eqref{eqn. big diagram of dg cats}. Let us explain why.
Because $\dgcatt$ has a cofibrant replacement functor, Proposition~\ref{prop. cofibrant localization is full localization} allows us to find a homotopy inverse to the map $\dgcat^\circ[\quis^{-1}] \to \dgcat[\quis^{-1}]$. In fact, the entire bottom triangle in~\eqref{eqn. big diagram of dg cats} becomes a diagram of equivalences of $\infty$-categories as we explain in Remark~\ref{remark. homcof equivalence}. 
\end{remark}

\begin{remark}
\label{remark. homcof equivalence}
Because, after localization along quasi-equivalences, cofibrant replacement is naturally equivalent to the identity functor, we have a diagram of $\infty$-categories
	\eqnn
	\xymatrix{
	\dgcatt_{\homcof}[\eqs^{-1}] \ar[d] \ar[drr]^-{\id}
	\\
	\dgcatt[\eqs^{-1}] \ar[r]^-Q_-{\simeq} 
		& \dgcatt^\circ[\eqs^{-1}]	\ar[r]
		& \dgcatt_{\homcof}[\eqs^{-1}].
	}
	\eqnnd
Together with the equivalence~\eqref{eqn. comparison map}, we see that the map $\dgcatt[\eqs^{-1}] \to \dgcatt_{\homcof}[\eqs^{-1}]$ admits both a left and a right inverse. Thus, the natural cofibrant replacement map induces an equivalence 
	\eqnn
	\dgcatt[\eqs^{-1}] \to \dgcatt_{\homcof}[\eqs^{-1}].
	\eqnnd
A similar argument shows that the induced functors to/from $\dgcatt_{\homproj}[\eqs^{-1}]$ are also equivalences.
\end{remark}

The following result is due to Haugseng. We refer to his work for details.
\begin{theorem}[Corollary 5.7 of~\cite{haugseng-rectification-enriched}]
\label{theorem. haugseng rectification}
The map~\eqref{eqn. comparison map} 
is an equivalence of $\infty$-categories.
\end{theorem}

As a result, every arrow in~\eqref{eqn. big diagram of oo cats of dg cats} is an equivalence of $\infty$-categories. So each of the $\infty$-categories appearing in~\eqref{eqn. big diagram of oo cats of dg cats} is equivalent to $\dgcat$.

On the other hand, it is a result of Pascaleff~\cite{pascaleff} and of Canonaco-Ornaghi-Stellari~\cite{cos-over-rings} that the natural inclusion $\dgcatt \to \Ainftystr$ induces an equivalence of $\infty$-categories
	\eqnn
	\dgcat:= \dgcatt[\eqs^{-1}] \to \Ainftystr[\eqs^{-1}] =: \Ainftycat^{\str}.
	\eqnnd
Theorem~\ref{theorem. main theorem} states that the natural inclusion $\Ainftystr \to \Ainfty$ induces an equivalence of $\infty$-categories
	\eqnn
	\Ainftycat^{\str} \to \Ainftycat.
	\eqnnd
We conclude:

\begin{theorem}
The inclusion $\dgcatt \to \Ainfty$ (into the category of not-necessarily-strictly unital $A_\infty$-categories and unital $A_\infty$-functors) induces an equivalence of $\infty$-categories
	\eqnn
	i: \dgcat \to \Ainftycat.
	\eqnnd
\end{theorem}

\begin{remark}
\label{remark. Ainftycat is presentably symmetric monoidal}
We do not endeavor here to create point-set models of homotopy (co)limits and homotopy-coherent tensor products of (not necessarily strictly) unital $A_\infty$-categories. But via the equivalence $i: \dgcat \to \Ainftycat$, we may conclude that $\Ainftycat$ enjoys all the formal properties of $\dgcat$ (such as admitting all limits and colimits and being presentable). Further, we use $i$ to transfer the symmetric monoidal structure of $\dgcat$ (which we will review shortly) to $\Ainftycat$. 

We conclude that this endows $\Ainftycat$ with the structure of a presentably symmetric monoidal $\infty$-category.
\end{remark}

\subsection{Monoidal structure on enrichments and on $\dgcat$}
\label{section. monoidal structures on enriched categories}
Let $\cV$ be a symmetric monoidal $\infty$-category. 
Gepner-Haugseng define a symmetric monoidal tensor product on $\inftyCat^{\cV}$. In fact, $\inftyCat^{\cV}$ is obtained as a localization of $\algcat(\cV)$ (and the monoidal structure is induced by this localization), so for most arguments (as ours here) it will suffice to understand the monoidal structure on $\algcat(\cV)$ as laid out in ibid.

Parsing the discussions there,\footnote{
For details the reader may consult Proposition~3.6.14,
(the discussion before) Lemma~3.6.15,
Proposition 4.3.11, 
(the proof of) Lemma~5.7.10,
and
Proposition~5.7.14 of~\cite{gepner-haugseng}.} we arrive at the following informal description of the symmetric monoidal structure $\tensor$ on $\algcat(\cV)$. We have that  
\enum[(i)]
	\item $\ob(A \tensor B) \simeq \ob A \times \ob B$,  and
	\item $\hom_{A \tensor B}( (a,b),(a',b')) \simeq \hom_A(a,a') \tensor^{\cV} \hom_B(b,b')$. That is, morphism spaces are determined using the composition 
		\eqnn
		\xymatrix{
		(\ob A \times \ob A) \times (\ob B \times \ob B) \ar[rr]^-{\hom_A \times \hom_B}
		&& \ob \cV \times \ob \cV \ar[r]^-{\tensor^{\cV}}
		& \ob \cV
		}
		\eqnnd
	where the last arrow is the monoidal product of $\cV$. 
\enumd
From the above description (because direct products of discrete spaces are discrete) it is clear that $\algcat(\cV)_{\Set}$ is a symmetric monoidal subcategory of $\algcat(\cV)$.

\begin{example}
\label{example. usual tensor of enriched categories}
When $\cV$ is an ordinary category, and when restricted to the subcategory $\algcat(\cV)_{\Set}$, the above endows $\algcat(\cV)_{\Set}$ with the usual tensor product of $\cV$-enriched categories. (To verify this fact is the only reason we explicated the monoidal structure on $\algcat(\cV)$.)

In particular, the symmetric monoidal structure on 
	\eqnn
	\algcat(\chain)_{\Set}, \qquad \algcat(\chain^{\circ})_{\Set},
	\qquad\text{and}\qquad
	\algcat(\chain^{\homproj})_{\Set}
	\eqnnd
is the usual tensor product of dg-categories: $\ob(A \tensor B) = \ob A \times \ob B$, and $\hom_{A \tensor B}( (a,b),(a',b')) = \hom_A(a,a') \tensor_{\kk} \hom_B(b,b')$.
\end{example}

\begin{remark}
The assigments $\cV \mapsto \algcat({\cV})_{\Set}$ and $\cV \mapsto \algcat({\cV})$ are functorial for symmetric monoidal $\cV$ and symmetric monoidal functors between them (Proposition 4.3.11 of~\cite{gepner-haugseng}).
\end{remark}

There are now two potential symmetric monoidal structures on $\dgcat$:
\enum[(i)]
	\item A symmetric monoidal structure induced by the localization
		\eqnn
		\algcat(\chain^\circ[\quis^{-1}])_{\Set} \to
		\algcat(\chain^\circ[\quis^{-1}])_{\Set}[\FFES^{-1}].
		\eqnnd
	(This is the symmetric monoidal structure guaranteed by the work of Gepner-Haugseng. Note we could use $\chain^{\homproj}$ in place of $\chain^{\circ}$ and obtain an equivalent symmetric monoidal structure.)
	\item A symmetric monoidal structure induced by the localization
		\eqnn
		\dgcatt_{\homcof} \to
		\dgcatt_{\homcof}[\eqs^{-1}]
		\eqnnd
	and pulling back to $\dgcat$ using the equivalence of Remark~\ref{remark. homcof equivalence} and composing with~\eqref{eqn. comparison map}. (We could use $\dgcatt_{\homproj}$ instead of $\dgcatt_{\homcof}$ and obtain an equivalent symmetric monoidal structure.)
\enumd

\begin{prop}
The above two symmetric monoidal structures are equivalent.
\end{prop}

\begin{proof}
In~\eqref{eqn. big diagram of oo cats of dg cats}, the two arrows labeled by $\tensor$ have natural symmetric monoidal structures induced by localization. 

(This is because the corresponding arrows in~\eqref{eqn. big diagram of dg cats} are symmetric monoidal functors. We note that the other arrows in~\eqref{eqn. big diagram of dg cats} do not induce symmetric monoidal structures on the maps between localizations -- $\dgcatt^\circ$ is not a monoidal subcategory of $\dgcatt$, and the tensor product of $\dgcatt$ does not preserve quasi-equivalences.) 

Aside from the domains and codomains of those two arrows, all $\infty$-categories in~\eqref{eqn. big diagram of oo cats of dg cats} have symmetric monoidal structures defined by pulling back symmetric monoidal structures through the equivalences in~\eqref{eqn. big diagram of oo cats of dg cats} that are not labeled by $\tensor$ symbols. (These induced structures are well-defined up to symmetric monoidal equivalence, because~\eqref{eqn. big diagram of oo cats of dg cats} is a coherent diagram of $\infty$-categories).

Because the composition of the two $\tensor$-labeled symmetric monoidal functors in~\eqref{eqn. big diagram of oo cats of dg cats} is underlied by an equivalence of $\infty$-categories by Theorem~\ref{theorem. haugseng rectification}, the result follows by the universal property of symmetric monoidal structures induced by localizations (Proposition~\ref{prop. localization is symmetric monoidal}~\eqref{item. universal property of symm monoidal localization}).
\end{proof}

\subsection{Identification with the derived tensor product}
By Section~\ref{section. monoidal structures on enriched categories}, we have a canonical symmetric monoidal structure on $\dgcat$ induced by the symmetric monoidal structure of (cofibrant or homotopically projective) chain complexes. This induces a symmetric monoidal structure on $\dgcatt[\eqs^{-1}]$ via the equivalence~\eqref{eqn. comparison map}.  On the other hand, we have a classical derived tensor product on the homotopy category $\ho \dgcatt[\eqs^{-1}]$. We now equate these monoidal structures.

\begin{prop}\label{prop. monoidal structure is derived dg tensor}
The monoidal structure on $\dgcatt[\eqs^{-1}]$ from Section~\ref{section. monoidal structures on enriched categories} is, at the level of the homotopy category, naturally equivalent to the derived tensor product of dg-categories (in the sense of To\"en).
\end{prop}

\begin{proof}
For this proof only, we abuse notation and let $Q: \dgcatt \to \dgcatt_{\homcof}$ denote the cofibrant replacement functor with codomain $\dgcatt_{\homcof}$. $Q$ induces an equivalence $\unQ$ of localized $\infty$-categories (Remark~\ref{remark. homcof equivalence}).

Consider the following diagram of $\infty$-categories:
	\eqnn
	\xymatrix{
		& \dgcatt[\eqs^{-1}] \times \dgcatt[\eqs^{-1}] \ar@{-->}[dd]
	\\
	\dgcatt \times \dgcatt  \ar@{.>}[ur] \ar[dd]^{Q \times \id}
	\\
		& \dgcatt_{\homcof}[\eqs^{-1}] \times \dgcatt[\eqs^{-1}] \ar@{-->}[rr] \ar@{-->}[dd]
		&& \dgcatt[\eqs^{-1}] \ar@{-->}[dd]^{\unQ}
	\\
	\dgcatt_{\homcof} \times \dgcatt \ar[rr]^{\tensor} \ar@{.>}[ur] \ar[dd]^{\id \times Q}
		&& \dgcatt \ar[dd] \ar@{.>}[ur]^-{L}
	\\
		& \dgcatt_{\homcof}[\eqs^{-1}] \times \dgcatt_{\homcof}[\eqs^{-1}] \ar@{-->}[rr]^{\underline{\tensor}}
		&& \dgcatt_{\homcof}[\eqs^{-1}]
	\\
	\dgcatt_{\homcof} \times \dgcatt_{\homcof} \ar[rr]^-{\tensor} \ar@{.>}[ur] 
		&& \dgcatt_{\homcof} \ar@{.>}[ur] 
		}
	\eqnnd
In the foreground of the above diagram are solid arrows. These solid arrows are functors of ordinary categories. We warn that the solid rectangle in the foreground does {\em not} commute (and in particular, the solid three-dimensional rectangular diagram has no interior coherence that we will utilize), but it commutes up to two natural transformations. Namely, for every pair $(A,B)$ of dg-categories we have natural maps
	\eqn\label{eqn. QAB naturality}
	Q(A \tensor B) \to A \tensor B \leftarrow A \tensor Q(B). 
	\eqnd
The lefthand map is always a quasi-equivalence, and when $A \in \ob \dgcatt_{\homcof}$, so is the righthand map. 

The dashed arrows in the diagram are all maps induced from the solid arrows by the formal process of $\infty$-categorical localization, with the diagonally oriented dotted arrows being the natural map from a category to its ($\infty$-categorical) localization. (Note that we have used that localization commutes with products.) Importantly, because the natural maps in~\eqref{eqn. QAB naturality} are quasi-equivalences, the dashed rectangle in the background {\em does} commute (up to homotopy). 

In what follows, we let $L: \dgcatt \to \dgcatt[\eqs^{-1}]$ denote the localization map (this is also labeled in the diagram). Because $L$ can be modeled to be a bijection at the level of objects, we let $LA$ denote the dg-category $A$ considered as an object in the localization.

The symmetric monoidal structure on $\dgcatt[\eqs^{-1}]$ induced by the equivalence $\unQ: \dgcatt[\eqs^{-1}]  \to \dgcatt_{\homcof}[\eqs^{-1}]$ has binary part, up to natural equivalence, given by
	\eqnn
	\unQ^{-1} \circ \underline{\tensor} \circ (\unQ \times \unQ)
	\eqnnd
where $\underline{\tensor}$ is the symmetric monoidal structure on $\dgcatt_{\homcof}[\eqs^{-1}]$ induced by the usual tensor product of dg-categories. The dashed diagram in the background, along with the homotopy-commutativity of the topmost rectangles containing the dotted diagonal arrows, tells us that the above composition is homotopic to a functor which, at the level of objects, is given by
	\eqnn
	(LA,LB) \mapsto L(Q(A) \tensor B).
	\eqnnd
This is precisely one formula for the derived tensor product of dg-categories. 

The above discussion did not take care to verify all higher coherences (we only considered the binary term, and only up to unspecified natural equivalences). Thankfully, the discussion is sufficient to conclude that (at the level of homotopy categories) the monoidal functor on $\ho \dgcatt[\eqs^{-1}]$ is naturally isomorphic to the assignment $(LA,LB) \mapsto L((Q(A)) \tensor B)$.
\end{proof}

\subsection{$\Ainftycat$ has internal homs}
\label{section. internal homs exist}
Let $\cC$ be a presentably monoidal $\infty$-category, so that for every object $C$, the functor $C \tensor - : \cC \to \cC$ preserves all colimits. Because $\cC$ is presentable, the adjoint functor theorem (Corollary 5.5.2.9 of~\cite{htt}) guarantees a right adjoint. Calling this right adjoint $\underline{\hom}(C,-)$, we witness a natural equivalence of Kan complexes
	\eqn\label{eqn. internal hom}
	\hom_{\cC}(C \tensor D, E)
	\simeq
	\hom_{\cC}(D, \underline{\hom}(C,E)).
	\eqnd
As usual, we will refer to $\underline{\hom}(C,E)$ as an internal hom object of $\cC$.
Note also that the unit of the monoidal structure on $\cC$ determines a monoidal and colimit-preserving functor
	\eqn\label{eqn. presentable unit}
	\spaces \to \cC
	\eqnd
characterized by sending the Kan complex $\Delta^0$ to the unit $1_{\cC}$.

\begin{prop}
Let $\cV$ be the $\infty$-category of chain complexes over a commutative ring $\kk$ (Definition~\ref{defn. oo cat of chain}).
The functor~\eqref{eqn. presentable unit} is induced by the functor of 1-categories sending a simplicial set to its normalized complex of singular chains over $\kk$.
\end{prop}

\begin{proof}
It is well-known that the functor of 1-categories $\sset \to \chain^\circ$ sending a simplicial set to its normalized $\kk$-linear chain complex is lax symmetric monoidal.\footnote{See Section 1 of~\cite{may-operads-and-sheaf-cohomology}, or VIII.8 of~\cite{maclane-homology}, or  2.5.7.12 of~\cite{kerodon}, for example.} Moreover, this functor sends weak homotopy equivalences to quasi-isomorphisms -- so by localizing both 1-categories along these classes of weak equivalences, one obtains\footnote{This follows from the paragraph preceding 4.1.3.4 in~\cite{higher-algebra}.} a lax symmetric monoidal functor $\spaces\to \chain^\circ[\quis^{-1}]$.  It is the content of the Eilenberg-Zilber Theorem that all the lax morphisms $C_*(X) \tensor C_*(Y) \to C_*(X \times Y)$ are quasi-isomorphisms, so in fact this lax symmetric monoidal functor is a symmetric monoidal functor upon passage to the localizations. Noting that $X \mapsto C_*(X)$ preserves homotopy colimits and sends $\Delta^0$ to the monoidal unit of $\cV$, we have verified the characterizing property of~\eqref{eqn. presentable unit}.
\end{proof}

\begin{remark}
\label{remark. presentably monoidal structure on V-enriched cats}
It is a result of Gepner-Haugseng (Proposition~5.7.16 of~\cite{gepner-haugseng}) that if $\cV$ is presentably monoidal, then so is the $\infty$-category $\inftyCat^{\cV}$ of $\cV$-enriched $\infty$-categories. In particular, as in~\eqref{eqn. internal hom}, one may contemplate the internal hom objects in $\cC = \inftyCat^{\cV}$. 
\end{remark}

Let $\cV$ be the $\infty$-category of chain complexes over $\kk$ (Definition~\ref{defn. oo cat of chain}). Then we have already seen in Section~\ref{section. equivalent dg cat models} the equivalences
	\eqnn
	\Ainftycat
	\simeq
	\dgcatt_{\homcof}[\eqs^{-1}] \simeq \inftyCat^{\cV} =\dgcat 
	\eqnnd
We have also witnessed  in Section~\ref{section. monoidal structures on enriched categories} that the middle equivalence can be made symmetric monoidal for the natural tensor products on $\inftyCat^{\cV}$ and $\dgcatt_{\homcof}[\eqs^{-1}]$. We endow $\Ainftycat$ with the tensor product induced by the above equivalence. By Remark~\ref{remark. presentably monoidal structure on V-enriched cats}, $\Ainftycat$ is presentably symmetric monoidal and hence has internal homs.

\begin{remark}[The underlying $\infty$-category]
Because~\eqref{eqn. presentable unit} is monoidal, it induces a functor from $\spaces$-enriched $\infty$-categories (which is to say, $\infty$-categories, by Example~\ref{example. infinity categories as enriched}) to $\cV$-enriched $\infty$-categories:
	\eqn\label{eqn. free dg category}
	\inftyCat \to \inftyCat^{\cV}.
	\eqnd
Moreover, because~\eqref{eqn. presentable unit} is a functor of presentably monoidal categories, so is~\eqref{eqn. free dg category} (Proposition~5.7.8 of~\cite{gepner-haugseng}). In particular,~\eqref{eqn. free dg category} has a right adjoint. We call this right adjoint the {\em underlying $\infty$-category} functor.

We caution that Gepner-Haugseng uses a slightly different definition for the underlying $\infty$-category functor. 
\eqref{eqn. free dg category} arises from a functor
	\eqn\label{eqn. free algcat}
	\algcat(\spaces) \to \algcat(\cV)
	\eqnd
obtained by applying~\eqref{eqn. presentable unit} on morphism spaces (Example~4.3.20 of~\cite{gepner-haugseng}).  When $\cV$ is presentably symmetric monoidal, \eqref{eqn. free algcat} is also. And the right adjoint to \eqref{eqn. free algcat} is what is called the underlying $\infty$-category functor in  Definition~5.1.10 of~\cite{gepner-haugseng}.
But we may localize along the fully faithful and essentially surjective morphisms to arrive at \eqref{eqn. free dg category}, so the effects of \eqref{eqn. free dg category} and \eqref{eqn. free algcat} are identical on objects.

Moreover, because $\inftyCat^{\cV}$ may be obtained by localizing $\algcat(\cV)$ or by localizing the full subcategory $\algcat(\cV)_{\Set}$ -- see~\eqref{eqn. two presentations for enriched cats} -- the underlying $\infty$-category functor may also be computed by taking the right adjoint to
	\eqn\label{eqn. free algcat set}
	\algcat(\spaces)_{\Set} \to \algcat(\cV)_{\Set}.
	\eqnd
(Note that \eqref{eqn. free algcat set} arises as the restriction of \eqref{eqn. free algcat} to $\algcat(\spaces)_{\Set}$.)
One may then localize~\eqref{eqn. free algcat set} along the fully faithful and essentially surjective morphisms. This is how we will compute (the right adjoint to) \eqref{eqn. free dg category}.

(Note also that the localizations of \eqref{eqn. free algcat} and of \eqref{eqn. free algcat set} are naturally identified thanks to the equivalence~\eqref{eqn. two presentations for enriched cats}.)

When $\cV$ is the $\infty$-category of chain complexes, we identify the right adjoint to  \eqref{eqn. free dg category} in Proposition~\ref{prop. dg nerve is right adjoint} as induced by a well-known point-set construction called the dg nerve. (The dg nerve is due to Lurie -- see Construction 1.3.1.6 of~\cite{higher-algebra}.)

\end{remark}

\subsection{Statement of theorems}

\begin{defn}[Split units]
Let $A$ be a strictly unital $A_\infty$-category. We say that the units of $A$ {\em split} or that {\em $A$ has split units}, if for every object $X$ the map $\kk \to \hom_A(X,X)$ sending $1_\kk$ to the strict unit $u_X$  splits as a map of graded $\kk$-modules (not necessarily as a map of complexes).
\end{defn}

\begin{theorem}
\label{theorem. internal homs in dgcat}
Let $A$ and $B$ be dg-categories. Then the following dg-categories are  equivalent in $\dgcat$:
	\begin{enumerate}[(a)]
	\item\label{item. internal hom in dgcat} 
		The internal hom object $\underline{\hom}_{\dgcat}(A,B)$ in the $\infty$-category of dg-categories.
	\item\label{item. internal hom in dgcat rqr cofib modules} 
		For any quasi-equivalence of dg-categories $A' \to A$ with $A'$ a dg-category with homotopically projective morphism complexes, the dg-category $h-proj((A')^{\op} \tensor B)^{rqr})$ of right quasirepresentable homotopically projective bimodules. 
	\item\label{item. internal hom in dgcat unital cofibrant Aoo functors} For any  $A' \to A$ as above, the dg-category $\fun_{A_\infty}(A',B)$ of unital $A_\infty$-functors from $A'$ to $B$.
	\item\label{item. internal hom unital projective A'} For any quasi-equivalence of $A_\infty$-categories $A'' \to A$ with $A''$ a unital $A_\infty$-category with homotopically projective morphism complexes, the dg-category $\fun_{A_\infty}(A'',B)$ of unital $A_\infty$-functors from $A''$ to $B$.
	\item\label{item. internal hom in dgcat strictly unital split unit Aoo functors} For any quasi-equivalence $A' \to A$ for which $A'$ is a dg-category with split units and  homotopically projective morphism complexes, the dg-category $\fun_{A_\infty}^{\str}(A',B)$ of strictly unital $A_\infty$-functors from $A'$ to $B$.
	\item\label{item. internal hom unital projective split units A'}  For any  strictly unital quasi-equivalence of $A_\infty$-categories $A'' \to A$ with $A''$ a strictly unital $A_\infty$-category with homotopically projective morphism complexes and split units, the dg-category $\fun_{A_\infty}^{\str}(A'',B)$ of strictly unital $A_\infty$-functors from $A''$ to $B$. 
	\enumd
The techniques of the proof will show that all equivalences except those involving \eqref{item. internal hom unital projective A'} and \eqref{item. internal hom unital projective split units A'} can be chosen to be natural in the homotopy category $\ho \dgcat$. (See Remark~\ref{remark. yoneda not natural yet}.)
\end{theorem}

We also compute mapping spaces. In what follows, $N$ denotes the dg-nerve, and for an $\infty$-category $X$, $X^\sim$ denotes the largest Kan complex inside $X$.

\begin{theorem}
\label{theorem. maps in dgcat}
Let $A$ and $B$ be dg-categories. Then the following Kan complexes are homotopy equivalent:
	\begin{enumerate}[(a)]
	\item\label{item. hom space in dgcat} The mapping space $\hom_{\dgcat}(A,B)$.
	\item\label{item. hom space in dgcat rqr cofib modules}  For any quasi-equivalence of dg-categories $A' \to A$ with $A'$ a dg-category with homotopically projective morphism complexes, the space $N(h-proj((A')^{\op} \tensor B)^{rqr})^{\sim}$ of right quasirepresentable homotopically projective bimodules. 
	\item\label{item. hom space in dgcat unital cofibrant Aoo functors} For any  $A' \to A$ as above, the space $N(\fun_{A_\infty}(A',B))^{\sim}$ of unital $A_\infty$-functors from $A'$ to $B$.	
	\item\label{item. hom space unital projective A'} For any quasi-equivalence of $A_\infty$-categories $A'' \to A$ with $A''$ a unital $A_\infty$-category with homotopically projective morphism complexes, the space $N(\fun_{A_\infty}(A'',B))^{\sim}$ of unital $A_\infty$-functors from $A''$ to $B$.
	\item\label{item. hom space in dgcat strictly unital split unit Aoo functors} For any quasi-equivalence $A' \to A$ for which $A'$ is a dg-category with split units and  homotopically projective morphism complexes, the space $N(\fun_{A_\infty}^{\str}(A',B))^{\sim}$ of strictly unital $A_\infty$-functors from $A'$ to $B$.
	\item\label{item. hom space unital projective split units A'}  For any  strictly unital quasi-equivalence of $A_\infty$-categories $A'' \to A$ with $A''$ a strictly unital $A_\infty$-category with homotopically projective morphism complexes and split units, the space $N(\fun_{A_\infty}^{\str}(A'',B))^{\sim}$ of strictly unital $A_\infty$-functors from $A''$ to $B$.
	\enumd
All equivalences except those involving \eqref{item. hom space unital projective A'} and \eqref{item. hom space unital projective split units A'} can be chosen to be natural with respect to arrows in the homotopy category $\ho \dgcat$. (See Remark~\ref{remark. yoneda not natural yet}.)
\end{theorem}

Below are the $A_\infty$ analogues:

\begin{theorem}
\label{theorem. internal homs in Aoocat}
Let $A$ and $B$ be unital $A_\infty$-categories. Then the following $A_\infty$-categories are equivalent in $\Ainftycat$:
	\begin{enumerate}[(a)]
	\item\label{item. internal hom in Aoocat} The internal hom $\underline{\hom}_{\Ainftycat}(A,B)$ in the $\infty$-category of $A_\infty$-categories.
	\item\label{item. Aoo internal hom unital projective A'} For any quasi-equivalence of $A_\infty$-categories $A' \to A$ with $A'$ a unital $A_\infty$-category with homotopically projective morphism complexes, the $A_\infty$-category $\fun_{A_\infty}(A',B)$ of unital $A_\infty$-functors from $A'$ to $B$.  
	\item\label{item. Aoo internal hom unital projective split units A'}  Supposing $B$ is strictly unital, and fixing any  strictly unital quasi-equivalence of $A_\infty$-categories $A' \to A$ with $A'$ a strictly unital $A_\infty$-category with homotopically projective morphism complexes and split units: the $A_\infty$-category $\fun_{A_\infty}^{\str}(A',B)$ of strictly unital $A_\infty$-functors from $A'$ to $B$. 
	\enumd
\end{theorem}

As a corollary we obtain mapping space computations. Below, $N$ stands for the $A_\infty$ nerve (Recollection~\ref{recollection. Aoo nerve}).

\begin{theorem}
\label{theorem. mapping spaces in Aoocat}
Let $A$ and $B$ be unital $A_\infty$-categories. Then the following Kan complexes are homotopy equivalent:
	\begin{enumerate}[(a)]
	\item\label{item. maps in Aoocat} The space $\hom_{\Ainftycat}(A,B)$ in the $\infty$-category of $A_\infty$-categories.
	\item\label{item. Aoo maps unital projective A'} For any quasi-equivalence of $A_\infty$-categories $A' \to A$ with $A'$ a unital $A_\infty$-category with homotopically projective morphism complexes, the space $N(\fun_{A_\infty}(A',B))^{\sim}$ of unital $A_\infty$-functors from $A'$ to $B$.  
	\item\label{item. Aoo maps unital projective split units A'} Supposing $B$ is strictly unital, and fixing any  strictly unital quasi-equivalence of $A_\infty$-categories $A' \to A$ with $A'$ a strictly unital $A_\infty$-category with homotopically projective morphism complexes and split units: the space $N(\fun_{A_\infty}^{\str}(A',B))^{\sim}$ of strictly unital $A_\infty$-functors from $A'$ to $B$. 
	\enumd
\end{theorem}

\begin{remark}
\label{remark. yoneda not natural yet}
The reader will note that the term ``naturally'' is missing in the statement of Theorem~\ref{theorem. internal homs in Aoocat} and \ref{theorem. mapping spaces in Aoocat}. This is because our proof invokes the Yoneda embedding $A \to \cY(A)$ for $A_\infty$-categories, and the naturality of this embedding does not yet seem to appear in the literature. This is the same reason that naturality was not attained in some of the equivalences in Theorem~\ref{theorem. internal homs in dgcat} and \ref{theorem. maps in dgcat}.

To orient the reader further, let us say that all internal-hom results about $\Ainftycat$ are being bootstrapped from the results about internal homs in $\dgcat$. To do this, one needs a way to turn $A_\infty$-categories into dg-categories, and this is why we invoke the Yoneda embedding.
\end{remark}

We will prove the above theorems in Section~\ref{section. internal hom computations}; we set up some notation and machinery in the interim.

\subsection{The dg nerve is the underlying $\infty$-category}
\begin{prop}\label{prop. dg nerve is right adjoint}
Let $\cV$ be the $\infty$-category of chain complexes over some commutative ring $\kk$. Then the right adjoint to~\eqref{eqn. free dg category} is (equivalent to the functor induced by) the dg nerve. 
\end{prop}

\begin{proof}
The dg nerve is a functor from the 1-category of dg-categories to the 1-category of simplicial sets, and is a right Quillen functor.\footnote{Proposition~1.3.1.20 of~\cite{higher-algebra}. Note that Lurie utilizes the Tabuada model structure there.} This Quillen adjunction thus determines\footnote{See Proposition~1.5.1 of~\cite{hinich-dwyer-kan-revisited} and Theorem~2.1 of~\cite{mazel-gee-quillen-adjunctions}.} an adjunction of $\infty$-categories
	\eqnn
	\LL: \inftyCat \iff \dgcat.
	\eqnnd
We wish to see that the left adjoint $\LL$  of this adjunction is equivalent to~\eqref{eqn. free dg category}. Let $\sset$ be the usual 1-category of simplicial sets and let $\sset \to \inftyCat = \sset[\eqs^{-1}]$ be the localization with respect to weak equivalences in the Joyal model structure. Letting $\underline{\Delta}$ denote the usual category of finite non-empty linear ordinals, both
	\eqnn
	\underline{\Delta} \to \sset
	\qquad\text{and the composition}\qquad
	\underline{\Delta} \to \sset \to \inftyCat
	\eqnnd
are fully faithful functors. $\inftyCat$ is generated by $\underline{\Delta}$ under colimits, so to see that the two left adjoints  $\LL$ and~\eqref{eqn. free dg category} are equivalent, it suffices to produce a natural equivalence between the composites
	\eqn\label{eqn. reduction of left adjoint to simplices}
	\underline{\Delta} \to \sset \to \inftyCat \xrightarrow{\LL} \dgcat
	\qquad
	\text{and}
	\qquad
	\underline{\Delta} \to \sset \to \inftyCat \xrightarrow{\eqref{eqn. free dg category}} \dgcat.
	\eqnd
The lefthand composite is easy to compute because we can compute the left adjoint to the dg nerve construction: The image $\LL[\Delta^k]$ of $\Delta^k$ is the dg category corepresenting the data of Construction~1.3.16 in~\cite{higher-algebra}.\footnote{We warn that loc. cit. uses homological conventions, while the present paper uses cohomological conventions.} 
 For every $i<j$, there exists a degree 0 morphism $x_{i,j} : i \to j$. $\hom_{\LL[\Delta^k]}(i,j)$ in degree 0 is a free $\kk$-module generated by words $x_{i_{n-1},i_n}\ldots x_{i_0,i_1}$ with $i_0=i,i_n=j$, $i_a < i_{a+1}$ and $n \geq 1$. For every $i \in \ob \Delta^k$ we think of $\hom(i,i) = \kk = \kk\cdot\emptyset_i$ as generated by an empty word $\emptyset_i$ representing the identity of $i$, and composition is given by concatenation of words. There are degree -1 elements that exhibit any two words generating $\hom_{\LL[\Delta^k]}(i,j)$ as cohomologous (and higher degree elements exhibiting higher coboundaries), arranged so that $\hom_{\LL[\Delta^k]}(i,j)$ is quasi-isomorphic to the chain complex $\kk$ concentrated in degree 0.

Consider the dg-category $\kk\tensor[k]$, where $\hom_{\kk \tensor [k]}(i,j)$ is $\kk$ if $i \leq j$ and zero otherwise, with the obvious composition $\kk \tensor \kk \xrightarrow{\cong} \kk$. Then
$\LL[\Delta^k]$ admits a natural (in the $\Delta^k$ variable) functor to $\kk \tensor [k]$ given by the identity on objects, sending all words to the generator of $\kk = \hom_{\kk \tensor [k]}(i_0,i_n)$, and sending all negative-degree elements to 0. This functor is a quasi-equivalence. (Indeed, one should think of $\LL[\Delta^k]$ as a resolution of $\kk \tensor [k]$.)

Now we seek to understand the righthand composite in~\eqref{eqn. reduction of left adjoint to simplices}. The $\infty$-category $\Delta^k$ is characterized as an $\infty$-category whose set of objects is in bijection with the poset $[k]$, and where $\hom_{\Delta^k}(i,j) \simeq pt$ is a contractible space when $i \leq j$ (and is empty otherwise), with the obvious composition law. To understand \eqref{eqn. free dg category}, we may thus apply~\eqref{eqn. presentable unit} on each morphism space -- and~\eqref{eqn. presentable unit} is modeled at the level of 1-categories as the normalized chain complex functor $C_*$ from the category of Kan complexes to $\chain$, then localizing along weak equivalences. Noting that $C_*(pt) \cong \kk$, we see that the righthand composite in~\eqref{eqn. reduction of left adjoint to simplices} is equivalent to the assignment sending $\Delta^k$ to $\kk \tensor [k]$. 
\end{proof}

\begin{recollection}
\label{recollection. Aoo nerve}
Recall also the notion of $A_\infty$ nerve, as constructed in~\cite{tanaka-thesis} and by Faonte in~\cite{faonte-nerve}. This is a functor from the category of strictly unital $A_\infty$-categories and strictly unital functors to the category of simplicial sets, with image landing in $\infty$-categories (i.e., quasi-categories). It is straightforward from the definitions (by choosing an appropriate sign convention) that the diagram
	\eqnn
	\xymatrix{
	\sset & \ar[l]_{N}  \dgcatt \ar[d]^i \\
	& \Ainftystr \ar[ul]^{N}
	}
	\eqnnd
commutes. (That is, if $A$ is a dg-category, its dg-nerve is equal to its $A_\infty$-nerve.) 
\end{recollection}

Because $i$ induces an equivalence of $\infty$-categories $\dgcat \to \Ainftycat$, Proposition~\ref{prop. dg nerve is right adjoint} thus implies:

\begin{prop}
\label{prop. A infty nerve is right adjoint}
The right adjoint to the composition
	\eqnn
	\xymatrix{
	\inftyCat \ar[r]^-{\eqref{eqn. free dg category}}
		& \dgcat \ar[r]^-{i}_-{\simeq}
		& \Ainftycat
	}
	\eqnnd
is, for strictly unital $A_\infty$-categories, modeled by the $A_\infty$ nerve.
\end{prop}

\begin{remark}[Nerves for unital $A_\infty$-categories]
When an $A_\infty$-category $A$ is (not necessarily strictly) unital, the nerve $N(A)$ produces a semisimplicial set -- this is because a given object $X \in \ob A$ has no distinguished single morphism deserving the title of $\id_X$. Regardless, $N(A)$ satisfies two special properties. (i) The weak Kan condition for semisimplicial sets, and (ii) every object admits an idempotent self-equivalence.  It is a result of Steimle that any such semisimplicial sets admits an enhancement to a simplicial set, unique up to homotopy; and~\cite{tanaka-units} explains how to make functors out of the resulting simplicial sets. Section~3.4 of~\cite{last-1} gives an account of the latter fact as well.

Because the results of this work do not require a point-set model for the putative nerve functor $\Ainfty \to \sset$ to the category of simplicial sets, we do not pursue a detailed explication for the unital case of the nerve.  (In fact, the $A_\infty$-nerve seems more natural to model as a functor to semisimplicial sets; the desired map to $\inftyCat$ is most easily modeled by localizing semisimplicial sets by a well-chosen class of morphisms to obtain a functor into $\inftyCat$).
\end{remark}

\begin{remark}
Let $A$ be a strictly unital $A_\infty$-category and $NA$ its $A_\infty$ nerve. By definition, the objects of $NA$ are the objects of $A$. We moreover have a natural isomorphism
	\eqnn
	H^{-i}\hom_A(x,y)
	\cong
	\pi_i \hom_{NA}(x,y)\qquad \forall i \geq 0.
	\eqnnd
One can verify this by hand straight from the definitions and from the usual Dold-Kan computations (see~\cite{tanaka-thesis}). Or, one can utilize Proposition~\ref{prop. A infty nerve is right adjoint} and bootstrap the result from the case of dg-categories, which is proven in Remark 1.3.1.12 of~\cite{higher-algebra}.
\end{remark}

\subsection{Generalities on homotopy categories}
\label{section. different models for homotopy cateogires}
Recall that if $\cC$ is a model category, one has a well-defined notion of homotopy category. It has (at least) four different, naturally equivalent definitions:
\enum
	\item\label{item. pi0 of localization} One takes the $\infty$-categorical localization $\cC[\eqs^{-1}]$, and takes the homotopy category of this localization. By construction, this definition produces a category with the same objects as $\cC$, and has morphisms given by $\pi_0 \hom_{\cC[\eqs^{-1}]}(X,Y)$. (This definition only depends on the weak equivalences of the model categorical structure.) 
	\item\label{item. naive localization} One takes the $1$-categorical localization of $\cC$ along $\eqs$.\footnote{See for example Gabriel-Zisman~\cite{gabriel-zisman}, 1.1 of Chapter 1.} By construction, this definition produces a category with the same objects as $\cC$, and has morphisms given by diagrams of finite-length zig-zags, modulo natural relations. While a priori such a definition may result in a category with non-small morphism classes (see~\cite{krause-localization-theory}, Example 4.15) model-categorical arguments allow us to conclude a posteriori that the morphism sets are (small) sets. (This definition only  depends on the weak equivalences of the model categorical structure.) 
	\item\label{item. subcategory localization} One takes the full subcategory of fibrant and cofibrant objects of $\cC$, then defines the set of morphisms from $X$ to $Y$ to be the quotient $\hom_{\cC}(X,Y)/{\sim}$ by the model-categorical relation of homotopy.  (This definition uses more than just the weak equivalences in the model category structure.) 
	\item\label{item. replacement localization} One defines the homotopy category to have the same collection of objects as $\cC$, then chooses for every object $X$ a fibrant-cofibrant replacement $RQX \leftarrow QX \to X$ (this need not be done functorially at the level of $\cC$) and defines morphisms as $\hom_{\cC}(RQX,RQY)/{\sim}$. It is an exercise that this results in a well-defined composition rule.  (This definition uses more than just the weak equivalences in the model category structure.) 
\enumd
	
\begin{remark}
The equivalence between \eqref{item. pi0 of localization} and \eqref{item. naive localization} is Proposition~3.1 of~\cite{dwyer-kan-calculating}, combined with the fact that hammock localization computes $\infty$-categorical localizations.\footnote{See Proposition~1.2.1 of~\cite{hinich-dwyer-kan-revisited}, where $\infty$-categorical localization is modeled using marked simplicial sets.}
This equivalence does not depend on $\cC$ being a model category -- this is relevant for us because $\Ainfty$ does not admit a model category structure.

As an aside, the equivalence holds even when the localizations may not be locally small. (One simply passes to a larger Grothendieck universe as needed to articulate categories with larger morphism sets.)

The equivalence between
\eqref{item. subcategory localization} 
and
\eqref{item. replacement localization} follows by noting that the former is (by definition) a full subcategory of the latter, and its inclusion is essentially surjective because weak equivalences are isomorphisms in the latter (Proposition~5.8 of~\cite{dwyer-spalinski}).

The equivalence between
\eqref{item. naive localization}
and
\eqref{item. replacement localization}
is Theorem~6.2 of~\cite{dwyer-spalinski}.
\end{remark}
%
%
%
%
%
%

\subsection{Whitehead style results}
\label{section. whitehead style}

\begin{prop}[Whitehead's Theorem for homotopically projective $A_\infty$-categories]
\label{prop. whitehead for a infty}
Let $A$ and $B$ be $A_\infty$-categories whose morphism complexes are homotopically projective chain complexes. If $f: A \to B$ is a quasi-isomorphism, then $f$ admits an inverse $A_\infty$-functor up to natural equivalence.
\end{prop}

\begin{proof}

For all objects $a,a' \in \ob A$, $f$ induces a quasi-isomorphism $f_{a,a'}: \hom_A(a,a') \to \hom_B(f(a),f(a'))$. By Proposition~\ref{prop. whitehead for projectives}, we conclude:
	\eqnn
	\text{each $f_{a,a'}$ admits a homotopy-inverse chain map.}
	\eqnnd
Because $f$ is essentially surjective, by the axiom of choice there exists a function $h: \ob B \to \ob A$ so that there exists an isomorphism $fh(b) \cong b$ in $H^0(B)$. Let $p: fhb \to b$ be a representative for this isomorphism and $r: b \to fhb$ be a representatives for the inverse; we can thus choose degree $-1$ elements $w \in \hom(b,b)$ and $v \in \hom(fhb,fhb)$ whose differential realizes the equalities $[m^2(p,r)] = [e_b]$ and $[m^2(r,p)] = [e_{fh(b)}]$ (here, $e_b$ is a homotopy unit for $b$). Let us abbreviate this simply as
	\eqnn
	\text{$f$ admits the data of $h$ and of $\{(p,r,w,v)\}_{b \in \ob B}$.}
	\eqnnd
The above two inline conditions are precisely the hypotheses laid out in Theorem~8.8 of~\cite{lyubashenko-category-of} for $f$ to admit an inverse functor $g: B \to A$ for which $fg$ is homotopic to (i.e., naturally equivalent to) $\id_B$ and $gf$ is homotopic to $\id_A$.
\end{proof}

\begin{prop}
\label{prop. pullback along equivalence is equivalence}
Suppose $f: A \to A'$ is a unital functor of unital $A_\infty$-categories admitting an inverse up to natural equivalence. Then for all unital $A_\infty$-categories $B$, the pullback
	\eqnn
	 \fun_{A_\infty}(A',B) \to \fun_{A_\infty}(A,B)
	\eqnnd
(between $A_\infty$-categories of unital functors) is also a unital functor of unital $A_\infty$-categories admitting an inverse up to natural equivalence.
\end{prop}

\begin{proof}
This follows from Section~6 of~\cite{lyubashenko-category-of}. Indeed, a repeated application of equation (6.1.1) of ibid. -- by doing work similar to the proof of Proposition 6.2 of ibid. -- shows that functors $f:A \to A'$, $g: A' \to A$ and natural equivalences $fg \sim \id_{A'}, gf \sim \id_A$ produce functors $f^*, g^*$ and natural equivalences $f^*g^* \sim \id_{\fun_{A_\infty}^{non-unital}(A',B)}$ and $g^* f^* \sim \id_{\fun_{A_\infty}^{non-unital}(A,B)}$. Note that, a posteriori, we may conclude that $g^*$ and $f^*$ are unital functors by Theorem~8.8 of ibid. Finally, because $f$ is assumed unital (and hence $g$ is unital), both $f^*$ and $g^*$ preserve the full subcategories $\fun_{A^\infty} \subset \fun_{A^\infty}^{non-unital}$ of unital $A_\infty$-functors. The restrictions of $f^*$ and $g^*$ to these full subcategories yields the result.

\end{proof}

\begin{prop}
\label{prop. homotopically projective abstract equivalences are equivalences}
Fix $A$ and $A'$, two unital $A_\infty$-categories with homotopically projective morphism complexes. Suppose one abstractly knows that $A$ and $A'$ are isomorphic in the homotopy category $\ho \Ainftycat$. Then there exists a unital functor of $A_\infty$-categories $A \to A'$ admitting an inverse up to natural equivalence.
\end{prop}

\begin{proof}
By the Yoneda Lemma, any unital $A_\infty$-category $A$ admits an $A_\infty$-functor $A \to \cY(A)$, where $\cY(A)$ is characterized as the image of $A$ under the Yoneda embedding into $\fun_{A_\infty}(A^{\op},\kk\Mod)$, and this functor is invertible up to natural equivalence.\footnote{See Corollary~A.9 of~\cite{lyubashenko-manzyuk-bimodules} and the references there.} Because $\cY(A)$ is a dg-category, it admits a cofibrant resolution $\tilde{A} \to \cY(A)$ in the model category of dg-categories. Because the maps $\dgcatt \to \Ainfty^{\str} \to \Ainfty$ induce equivalences\footnote{Combine Theorem~B of~\cite{cos-over-rings} and our Theorem~\ref{theorem. main theorem}.} 
upon localizing along quasi-equivalences, the induced functor $i: \dgcat \to \Ainftycat$ is an equivalence of $\infty$-categories. It follows that $\tilde{A}$ and $\tilde{A'}$ are abstractly isomorphic in $\ho \dgcat$. On the other hand, both of these dg-categories are cofibrant (and every object of $\dgcatt$ is fibrant), so by the equivalence
of~\eqref{item. pi0 of localization} and~\eqref{item. subcategory localization} in Section~\ref{section. different models for homotopy cateogires}, 
we conclude that there are morphisms in
	\eqnn
	\hom_{\dgcatt}(\tilde{A},\tilde{A'})
	\qquad
	\text{and}
	\qquad
	\hom_{\dgcatt}(\tilde{A'},\tilde{A})
	\eqnnd
which descend to be mutually inverse isomorphisms in the homotopy category $\ho\dgcat$. Pushing forward these morphisms along $i$, we obtain the middle $\leftrightarrow$ in the following collection of quasi-equivalences of $A_\infty$-categories:
	\eqnn
	A \leftrightarrow \cY(A) \leftarrow \tilde{A} \leftrightarrow \tilde{A'} \to \cY(A') \leftrightarrow A'.
	\eqnnd
Here, $\leftrightarrow$ denotes two quasi-equivalences that are mutually inverse up to natural equivalence. On the other hand, because $\tilde{A}$ and $A$ both have homotopically projective morphism complexes, the composite quasi-equivalence $A \leftarrow \tilde{A}$ admits a homotopy inverse up to natural equivalence (Proposition~\ref{prop. whitehead for a infty}). By composing the obvious arrows, one obtains a quasi-equivalence from $A$ to $A'$. This admits an inverse up to natural equivalences because both $A$ and $A'$ have homotopically projective morphism complexes (Proposition~\ref{prop. whitehead for a infty} again). Noting that if a functor between unital $A_\infty$-categories admits an inverse up to natural equivalence, it is automatic that the functor is unital (Corollary~8.9 of~\cite{lyubashenko-category-of}), the claim is proven. 

\end{proof}

\subsection{Computations of internal homs}
\label{section. internal hom computations}
\begin{proof}[Proof of Theorem~\ref{theorem. internal homs in dgcat}]
We have seen that $\dgcat$ is closed symmetric monoidal (Section~\ref{section. internal homs exist}). It follows that the homotopy category $\ho\dgcat$ is closed symmetric monoidal, due to the natural bijections
	\begin{align}
	\hom_{\ho \dgcat}(C \tensor_{\dgcat} D, E)
	&=   \pi_0\hom_{\dgcat}(C \tensor_{\dgcat} D, E) \nonumber \\
	& \cong \pi_0\hom_{\dgcat}(D, \underline{\hom}(C,E)) \nonumber \\
	&=   \hom_{\ho \dgcat}(D, \underline{\hom}(C,E)) \nonumber 
	\end{align}
induced by~\eqref{eqn. internal hom}. 

On the other hand, we have already verified in Proposition~\ref{prop. monoidal structure is derived dg tensor} that -- at the level of homotopy categories -- the assignment $(C,D) \mapsto C \tensor_{\dgcat} D$ is naturally isomorphic to the assignment $(C,D) \mapsto Q(C) \tensor D$, where $\tensor$ is the usual tensor product on the 1-category of dg-categories. Thus, $ \underline{\hom}(C,E)$ is naturally isomorphic in $\ho \dgcat$ to any dg-category $H_{C,E}$ exhibiting the closedness of $\ho \dgcat$ under the monoidal structure $(A,B) \mapsto Q(A) \tensor B$. 

$\eqref{item. internal hom in dgcat} \simeq \eqref{item. internal hom in dgcat rqr cofib modules}$. 
One such $H_{C,E}$ is identified by To\"en (and improved upon) by Canonaco-Stellari: The dg-category of cofibrant (or, more generally, homotopically projective) quasi right representable modules over $A^{\op} \tensor^{\LL} B$. (See Theorem 6.1 of~\cite{toen-homotopy-theory-of-dg-cats} and Theorem~1.1 of~\cite{canonaco-stellari-internal-homs-via-extensions}.)

$\eqref{item. internal hom in dgcat} \simeq \eqref{item. internal hom in dgcat unital cofibrant Aoo functors}$.
Another $H_{C,E}$ is identified by Canonaco-Ornaghi-Stellari~\cite{cos-over-rings} as the dg-category of unital $A_\infty$-functors from a homotopically projective replacement $A'$ of $A$, to $B$ (Theorem C of~\cite{cos-over-rings}).

$\eqref{item. internal hom in dgcat unital cofibrant Aoo functors}
\simeq
\eqref{item. internal hom unital projective A'}
$. 
By Proposition~\ref{prop. homotopically projective abstract equivalences are equivalences}, there exists a functor of $A_\infty$-categories $A' \to A''$ admitting an inverse up to natural equivalence. By Proposition~\ref{prop. pullback along equivalence is equivalence}, this yields a functor admitting an inverse up to natural equivalence $\fun_{A_\infty}(A',B) \simeq \fun_{A_\infty}(A'',B)$. Because these functor $A_\infty$-categories are dg-categories, we may identify them as in the images of the equivalence $i: \dgcat \to \Ainftycat$. Because $i$ is an equivalence of $\infty$-categories, we conclude that these functor dg-categories are equivalent in $\dgcat$ as well.

$\eqref{item. internal hom in dgcat unital cofibrant Aoo functors}
\simeq
\eqref{item. internal hom in dgcat strictly unital split unit Aoo functors}
$. The natural functor of dg-categories from $\fun_{A_\infty}^{\str}(A',B)$ to $\fun_{A_\infty}(A',B)$ is an equivalence if $A'$ is a dg-category with homotopically projective morphism complexes and split units (Remark 5.3 of~\cite{cos-over-rings}).

$\eqref{item. internal hom unital projective A'}
\simeq
\eqref{item. internal hom unital projective split units A'}
$. It suffices to prove that when $A''$ is strictly unital, then the inclusion $\fun_{A_\infty^{\str}}(A'',B) \to \fun_{A_\infty}(A'',B)$ of strictly unital functors into unital functors is an equivalence of $A_\infty$-categories. The inclusion is fully faithful by definition, so it remains to prove that any unital $A_\infty$-functor is homotopic to a strictly unital one. This is Lemma~4.2 of~\cite{cos-over-rings}.

\end{proof}

\begin{proof}[Proof of Theorem~\ref{theorem. internal homs in Aoocat}]
$\eqref{item. internal hom in Aoocat} \simeq \eqref{item. Aoo internal hom unital projective A'}$.
Let $A' \to A$ be a quasi-equivalence from a unital $A_\infty$-category with homotopically projective morphism complexes. This induces an isomorphism  of internal hom objects
	\eqnn
	\underline{\hom}_{\Ainftycat}(A,B)
	\cong
	\underline{\hom}_{\Ainftycat}(A',B) \in \ho \Ainftycat.
	\eqnnd
Choose $A_\infty$-functors $A' \to \cY(A'), B \to \cY(B)$ to dg-categories, admitting an inverse up to natural equivalence. (For example, by choosing the image of the Yoneda embedding as in the proof of Proposition~\ref{prop. homotopically projective abstract equivalences are equivalences}.) We obtain another isomorphism of internal hom objects
	\eqnn
	\underline{\hom}_{\Ainftycat}(A',B)
	\cong
	\underline{\hom}_{\Ainftycat}(\cY(A'),\cY(B)) \in \ho \Ainftycat.
	\eqnnd
On the other hand, because we have employed the symmetric monoidal structure on $\Ainftycat$ induced by the equivalence $i: \dgcat \to \Ainftycat$, we have an isomorphism between the internal hom object in $\dgcat$ and the internal hom object in $\Ainftycat$:
	\eqnn
	\underline{\hom}_{\Ainftycat}(\cY(A'),\cY(B))
	\cong
	i(\underline{\hom}_{\dgcat}(\cY(A'),\cY(B))) \in \ho\Ainftycat.
	\eqnnd
Now one may choose a homotopically projective replacement of dg-categories $\tilde{A'} \to \cY(A')$ and invoke Theorem~\ref{theorem. internal homs in dgcat}\eqref{item. internal hom in dgcat unital cofibrant Aoo functors} to conclude that
	\eqnn
	i(\underline{\hom}_{\dgcat}(\cY(A'),\cY(B))) 
	\cong
	i(\fun_{\Ainfty}(\tilde{A'},\cY(B))) 
	= \fun_{\Ainfty}(\tilde{A'},\cY(B))) \in \ho \Ainftycat.
	\eqnnd
To finish the proof, consider the homotopy inverse map $\cY(A') \to A'$ to the Yoneda equivalence, so that the composition $\tilde{A'} \to \cY(A') \to A'$ is a quasi-equivalence. Because both $\tilde{A'}$ and $A'$ have homotopically projective mapping complexes, the quasi-equivalence $\tilde{A'} \to A'$ admits an inverse up to natural equivalence (Proposition~\ref{prop. whitehead for a infty}.) Thus the functor categories out of $\tilde{A'}$ and out of $A'$ are equivalent as $A_\infty$-categories (Proposition~\ref{prop. pullback along equivalence is equivalence}). So we conclude
	\eqnn
	\fun_{\Ainfty}(\tilde{A'},\cY(B)))
	\cong 
	\fun_{\Ainfty}(A',\cY(B)))
	\cong 
	\fun_{\Ainfty}(A',B)
	\in \ho \Ainftycat
	\eqnnd
where the last isomorphism follows because $\cY(B) \to B$ is a functor invertible up to natural equivalence. We are finished by tracing through the above isomorphisms.

$\eqref{item. internal hom in Aoocat} 
\simeq
\eqref{item. Aoo internal hom unital projective split units A'} $.
The proof is identical to that of the case
$\eqref{item. internal hom unital projective A'}
\simeq
\eqref{item. internal hom unital projective split units A'}$
in Theorem~\ref{theorem. internal homs in dgcat}.

\end{proof}

\begin{proof}[Proof of Theorem~\ref{theorem. maps in dgcat} and Theorem~\ref{theorem. mapping spaces in Aoocat}.]
By setting $D$ to equal the unit dg-category $\kk \tensor [0]$, we have natural homotopy equivalences of Kan complexes
	\begin{align}
	\hom_{\dgcat}(C, E)
		& \simeq \hom_{\dgcat}(C \tensor_{\dgcat} D, E) \nonumber \\
		& \simeq \hom_{\dgcat}(D, \underline{\hom}(C,E)) \nonumber \\
		& \simeq \hom_{\inftyCat}(\Delta^0, N\underline{\hom}(C,E))) \nonumber \\
		& \simeq N(\underline{\hom}(C,E))^{\sim} \nonumber .
	\end{align}
Because the dg-nerve $N$ is a right Quillen functor, it respects quasi-equivalences among fibrant objects -- hence, among all objects (all objects are fibrant in the Tabuada model structure; see Proposition 2.3(1) of~\cite{toen-homotopy-theory-of-dg-cats}). So the above homotopy type is determined by the quasi-equivalence class of $\underline{\hom}(C,E)$. This proves Theorem~\ref{theorem. maps in dgcat}.

For the $A_\infty$-case (Theorem~\ref{theorem. mapping spaces in Aoocat}), we can write a similar proof by utilizing Proposition~\ref{prop. A infty nerve is right adjoint}.
\end{proof}

\subsection{Proof of Corollary~\ref{cor. cohomology and homotopy groups}}
\begin{proof}
By Theorem~\ref{theorem. mapping spaces in Aoocat} we have a homotopy equivalence
	\eqnn
	\hom_{\Ainftycat}(A',B)
	\simeq
	N(\fun_{A_\infty}(A',B))^\sim
	\eqnnd
and in particular an induced bijection on $\pi_0$ of both sides. By definition of the underlying $\infty$-groupoid $X^\sim$ of an $\infty$-category $X$, $\pi_0(X^\sim)$ is computed by relating any two vertices of $X$ when they are related by an equivalence in $X$. When $X$ is the nerve of $\fun_{A_\infty}(A',B)$, two functors $f,g: A' \to B$ are related if and only if there exists a natural equivalence from $f$ to $g$ -- i.e., if and only if $f$ and $g$ are isomorphic in the cohomology category $H^0(\fun_{A_\infty}(A',B))$. This shows the first claim.

When $f$ and $g$ are isomorphic in $H^0\fun_{A_\infty}(A',B)$, one has a homotopy equivalence of chain complexes
	\eqnn
	\hom_{\fun_{A_\infty}(A',B)}(f,g)
	\simeq
	\hom_{\fun_{A_\infty}(A',B)}(f,f).
	\eqnnd
So the rest of the claims are reduced to the case $f = g$.

We recall some generalities on the $A_\infty$-nerve.
First, if $C$ is an $A_\infty$-category, the space $\hom_{N(C)}(x,x)$ is homotopy equivalent to the Dold-Kan construction applied to the chain complex $\hom_C(x,x)$. See for example Proposition~2.3.12 of~\cite{tanaka-thesis}. Further, morphisms $f \in \hom^0_C(x,x)$ that are invertible up to homotopy are sent to edges $x \to x$ in $N(C)$ that are homotopy invertible. On the other hand, homotopy classes of such homotopy-invertible edges is one model for $\pi_1(N(C)^{\sim},x)$. So we find that
	\eqnn
	\pi_1(N(C)^{\sim},x)
	\cong
	H^0(\hom_C(x,x))^\times.
	\eqnnd
Finally, going back to the fact that the Dold-Kan construction applied to $\hom_C(x,x)$ is homotopy equivalent to $\hom_{N(C)}(x,x)$, and noting that the homotopy groups of the latter are isomorphic to those of $\hom_{N(C)^\sim}(x,x)$, we find isomorphisms
	\eqnn
	\pi_i \hom_{N(C)^{\sim}}(x,x)
	\cong
	H^{-i} \hom_C(x,x),
	\qquad
	i \geq 1
	\eqnnd
For any $\infty$-category $\cC$, the space $\hom_{\cC}(x,x)$ is homotopy equivalent to the based loop space of $\cC^\sim$ based at $x$, so we find
	\eqnn
	\pi_{i+1} N(C)^{\sim}(x,x)
	\cong
	H^{-i} \hom_C(x,x).
	\eqnnd
The claims of the Corollary follow by taking $C = \fun_{A_\infty}(A,B)$.
\end{proof}

\subsection{Proof of Corollary~\ref{cor. HH}}
\begin{proof}
This is immediate from Corollary~\ref{cor. cohomology and homotopy groups} by taking $B = A'$ and $f=g=\id_{A'}$. 
\end{proof}

\section{Appendix: Alternate proof of Lemma~\ref{lemma. j has left inverse}}
Here we provide an alternate proof of Lemma~\ref{lemma. j has left inverse} through explicit computation. 

\subsection{The ideal $\II$}
Fix an $A_\infty$-category $A$. Recall that an {\em $A_\infty$-ideal} $I$ of $A$ is a collection of graded $\kk$-modules
	\eqnn
	I(X,Y) \subset \hom_A(X,Y),
	\qquad
	X,Y \in \ob A
	\eqnnd
such that for all $k \geq 1$, $m^k(\bx_k,\ldots,\bx_1)$ is an element of $I$ whenever at least one of $\bx_i$ is in $I$. One may then define the quotient $A_\infty$-category $A/I$ to have the same objects as $A$, with morphism complexes
	\eqnn
	\hom_{A/I}(X,Y) := \hom_A(X,Y)/I(X,Y)
	\eqnnd
with $A_\infty$-operations induced by those of $A$. Given a collection of morphisms $K$ of $A$, the $A_\infty$-ideal generated by $K$ is the smallest $A_\infty$-ideal containing $K$. 

\begin{notation}[$\II$]
\label{notation. II}
Let $A$ be a strictly unital $A_\infty$-category. For every object $X \in \ob A$, let $u_X$ denote the strict unit of $X$. 
Consider the $A_\infty$-category $\Aploc$, and let $1_X \in \hom_{A^+}(X,X)$ denote the augmentation unit of $X$. We let $\II$ denote the $A_\infty$-ideal inside $\Aploc$ generated by the elements
	\eqnn
	1_X - u_X \in \hom_{A^+}(X,X) \subset \hom_{\Aploc}(X,X)
	\qquad
	X \in \ob A.
	\eqnnd
\end{notation}

In general it is difficult to concretely describe the ideal generated by a collection of elements. $\II$ admits a tractable description:

\begin{prop}
\label{prop. the ideal I}
For every pair of objects $X,Y \in \ob A$, the complex $\II(X,Y)$ is -- as a graded $\kk$-module -- spanned by: 
\begin{itemize}
\item  elements of length $1$ that are scalar multiples of $1_X - u_X$ when $X = Y$. (When $X \neq Y$, the length-1 filtered piece of $\II(X,Y)$ is zero.)
\item elements of length $l\geq 2$ of the form
	\eqn\label{eqn. words with units}
	\bx_l | \ldots |\bx_{i+1} | \bw | \bx_{i-1} | \ldots | \bx_1,
	\qquad 1 \leq i \leq l,
	\eqnd
where $\bw$ is some $\kk$-linear combination of
	\eqnn
	a \tensor (1_X - u_X),
	\qquad
	b \tensor (1_X - u_X),
	\qquad
	c \tensor (1_X - u_X),
	\qquad\text{and}\qquad
	d \tensor (1_X - u_X).
	\eqnnd
Here, we are invoking the notation from~\eqref{eqn. a b c d tensors} and~\eqref{eqn. abcd sum tensor}. When $i = 1$, $\bw$ of course contains no $b$ and no $d$ terms. Likewise when $i=l$,  $\bw$ contains no $c$ and no $d$ terms.
\end{itemize}
\end{prop}

\begin{proof}
Let $\JJ(X,Y)$ denote the graded $\kk$-module spanned by the elements described in the Proposition. We will first show that $\JJ(X,Y) \subset \II(X,Y)$.

Fix $W \in \ob A$ and a unit $e \in \hom_A(W,W)$. Let us compute $m^2_{ \Aploc}$ along the subcomplex
	\begin{align}
	\left(
	\hom_{\Tw A^+}(\cone(e),Y) | 
	\hom_{\Tw A^+}(X,\cone(e))\right)
	&\bigotimes
	\hom_{A^+}(X,X)
	\nonumber \\
	\subset
	\hom_{\Aploc}(X,Y)
	&\bigotimes
	\hom_{\Aploc}(X,X). \nonumber
	\end{align}
We have by definition~\eqref{eqn. operations in bar construction} that
	\begin{align}
	m^2_{\Aploc}\left(\bx_2 | \bx_1 \bigotimes (1_X - u_X)\right)
	&= m^3_{\Tw A^+} ( \bx_2 \tensor \bx_1 \tensor (1_X - u_X)) \nonumber \\
	& \qquad  \pm \bx_2 | m^2_{\Tw A^+}(\bx_1 \tensor (1_X - u_X))\nonumber \\
	&= \pm \bx_2 | m^2_{\Tw A^+}(\bx_1 \tensor (1_X - u_X))
	 \nonumber
	\end{align}
where the $m^3$ term vanishes because $1_X$ and $u_X$ are strict units (in $A^+$ and in $A$, respectively). Now, using the notation from~\eqref{eqn. a b c d tensors} and~\eqref{eqn. abcd sum tensor},  let us write
	\eqnn
	\bx_1 = a \tensor p_1 + c \tensor q_1. 
	\eqnnd
Then
	\eqnn
	 m^2_{\Tw A^+}(\bx_1 \tensor (1_X - u_X))
	 =
	 a \tensor m^2_{A^+}(p_1 \tensor (1_X - u_X))
	 \pm c \tensor m^2_{A^+}(q_1 \tensor (1_X - u_X)).
	\eqnnd
(We note that the higher $m^k$ terms in the definition of $m^2_{\Tw A^+}$ again vanish because $1_X,u_X$ are strict units.)
We see that this element is zero unless $W = X$; and even then, it is only sensitive to the $1_X$ components of $p_1$ and of $q_1$. In particular, 
	$m^2_{\Tw A^+}(\bx_1 \tensor (1_X - u_X))$
is a linear combination of terms of the form
	\eqn\label{eqn. m2 Tw of 1 - u}
	a \tensor (1_X - u_X)
	\qquad\text{and}\qquad
	c \tensor (1_X - u_X).
	\eqnd
We conclude that $m^2_{\Aploc}\left(\bx_2 | \bx_1 \bigotimes (1_X - u_X)\right)$ is a linear combination of words of the form
	\eqnn
	\bx_2 | (a \tensor (1_X - u_X))
	\qquad\text{and}\qquad
	\bx_2 | (c \tensor (1_X - u_X)).
	\eqnnd
A similar argument computing $m^2_{\Aploc}$ along
	\eqnn
	\hom_{A^+}(X,X)
	\bigotimes	\left(
	\hom_{\Tw A^+}(\cone(e),X) | 
	\hom_{\Tw A^+}(Y,\cone(e))\right)
	\eqnnd
shows that $\II$ also contains words of length 2 of the form
	\eqnn
	(a \tensor (1_X - u_X)) | \bx_1
	\qquad\text{and}\qquad
	(b \tensor (1_X - u_X)) | \bx_1.
	\eqnnd
Now note that when computing $m^k_{\Aploc}$, any $m^k_{\Tw A^+}$ terms containing a $1_X - u_X$ factor vanishes. Thus, by repeating the above arguments and noting that $cb = d$, we conclude that any element of $\JJ(X,Y)$ can be written as a linear combination of elements obtained by successively applying $A_\infty$-operations to elements of the form $1_X - u_X$. Thus $\JJ(X,Y) \subset \II(X,Y).$

It now suffices to show that $\JJ(X,Y)$ is an $A_\infty$-ideal. So fix objects $X^{(0)},\ldots, X^{(N)} \in \ob A$ and an element of the form
	\eqnn
	\bx^{(N)} \tensor \ldots \tensor \bx^{(1)}
	\in
	\hom_{\Aploc}(X^{(N-1)},X^{(N)}) \tensor \ldots \tensor \hom_{\Aploc}(X^{(0)},X^{(1)})
	\eqnnd
where we assume each $\bx^{(j)}$ is a single word of length $l_j$:
	\eqnn
	\bx^{(j)}
	=
	\bx^{(j)}_{l_j} | \ldots | 
	\bx^{(j)}_{1}
	\in
	\hom_{\Tw A^+}(C^{(j)}_{l_j-1}, X^{(j)}) |\ldots | \hom_{\Tw A^+}(X^{(j-1)}, C^{(j)}_1).
	\eqnnd
Here, for each $1 \leq k \leq l_j$, we have $C^{(j)}_k = \cone(e^{(j)}_k)$ for some unit $e^{(j)}_k \in \hom_A(W^{(j)}_k, W^{(j)}_k)$. We claim: If for some $i \leq N$, $\bx^{i} \in \JJ(X^{(i-1)},X^{(i)})$, then 
	\eqnn
	m^N_{\Aploc}(\bx^{(N)} \tensor \ldots \tensor \bx^{(1)}) \in \JJ(X^{(0)},X^{(N)}). 
	\eqnnd
Writing $m^N_{\Aploc}$ as the summation in~\eqref{eqn. operations in bar construction}, it suffices to check that the four terms
	\eqnn
	m^\beta_{\Tw A^+}(\ldots,a \tensor(1_X-u_X),\ldots),
	\qquad
	m^\beta_{\Tw A^+}(\ldots,b \tensor(1_X-u_X),\ldots),
	\eqnnd
	\eqnn
	m^\beta_{\Tw A^+}(\ldots,c \tensor(1_X-u_X),\ldots),
	\qquad
	m^\beta_{\Tw A^+}(\ldots,d \tensor(1_X-u_X),\ldots)
	\eqnnd
are all elements of $\JJ$. When $\beta \geq 3$, this is obvious because the above terms vanish: $1_X$ and $u_X$ are strict units (in $A^+$ and in $A$, respectively). When $\beta = 2$, the claim follows by repeating the computations surrounding~\eqref{eqn. m2 Tw of 1 - u}.

We are left only to check that $\JJ$ is closed under $m^1_{\Aploc}$. This follows by noting that (i) $m^1_{A^+}(1_X) = 0$ and $m^1_{A^+}(u_X) = 0$, and (ii) For any morphism $f$ in $A$ (and in particular, for any unit in $A$) we have:
	\begin{align}
	m^2_{A^+}(f,1_X - u_X) = f - f = 0 \nonumber \\
	m^2_{A^+}(1_X - u_X,f) = f - f = 0 \nonumber \\
	m^3_{A^+}(1_X - u_X,f,1_X - u_X) = 0 \nonumber .
	\end{align}
In particular, $m^1_{\Tw A^+}( z \tensor (1_X - u_X)) = 0$ when $z$ equals any of $a,b,c,d$.
\end{proof}

\subsection{A left inverse}
Suppose that $A$ and $B$ are strictly unital. Then any strictly unital functor $f: A \to B$ respects strict units, so the induced functor $A^+[\id_A^{-1}] \to B^+[\id_B^{-1}]$ respects the $A_\infty$-ideals $\II_A$ and $\II_B$ defined in Notation~\ref{notation. II}. In particular, the construction $A \mapsto \Aploc/\II_A$ is a functor from $\Ainftystr$ to itself.

\begin{notation}
$\tau_{/\II}$ denotes the functor (in the classical sense)
	\eqnn
	\tau_{/\II}:
	\Ainftystr \to \Ainftystr,
	\qquad
	A \mapsto \Aploc/\II.
	\eqnnd
\end{notation}

\begin{remark}
\label{remark. nat trans tau to tau mod I}
For any strictly unital functor $f: A \to B$, the induced diagram
	\eqn\label{eqn. naturality of mod I}
	\xymatrix{
	A^+[\id_A^{-1}] \ar[r]  \ar[d]
		& B^+[\id_B^{-1}]  \ar[d] \\
	A^+[\id_A^{-1}]/\II_A \ar[r] 
		& B^+[\id_B^{-1}]/\II_B 
	}
	\eqnd
commutes. Note all arrows above are strictly unital functors. So we have witnessed a natural transformation
	\eqnn
	\tau \circ j \to \tau_{/\II}.
	\eqnnd
\end{remark}

\begin{lemma}
\label{lemma. mod I is quasi equivalence}
For any strictly unital $A$, every map in the composition 
	\eqnn
	A \to A[\id_A^{-1}] \to A^+[\id_A^{-1}] \to  A^+[\id_A^{-1}]/\II
	\eqnnd
is a quasi-equivalence of $A_\infty$-categories.
\end{lemma}

\begin{proof}
Each map is a bijection on objects so it suffices to show that each map induces a quasi-isomorphism on all hom complexes.

Fix two objects $X,Y \in \ob A$.
Let 
	\eqnn
	F^+_{\leq l} \subset \hom_{\Aploc}(X,Y)
	\qquad\text{and}\qquad
	F_{\leq l} \subset \hom_{A[\id_A^{-1}]}(X,Y)
	\eqnnd
denote the subcomplexes of words of length $\leq l$. We let 
	\eqnn
	G_{\leq l}
	\subset
	\hom_{\Aploc}(X,Y)/\II(X,Y) 
	\eqnnd
denote the image of $F^+_{\leq l}$ in the quotient. Then for each $l \geq 1$, the associated graded complex $G_{\leq l}/G_{\leq l-1}$ is identified with the complex $F_{\leq l}/F_{\leq l-1}$. (This is a consequence of Proposition~\ref{prop. the ideal I}.)

On the other hand, for $l \geq 2$, the $l$th associated graded piece of $F_{\leq \bullet}$ is a direct sum of contractible chain complexes. After all, the summand 
	\eqnn
	\hom_{\Tw A}(C_{l-1},Y) | \hom_{\Tw A}(C_{l-2},C_{l-1}) | \ldots | \hom_{\Tw A}(X, C_1)
	\eqnnd
associated to a choice of units $e_1,\ldots,e_{l-1}$, is a mapping cone of a morphism of the form  
	\eqnn
	\pm m^2(-,e_{l-1}) | \id | \ldots | \id.
	\eqnnd
(See Example~\ref{example. hom of cones is a cone}.) By definition of units, $m^2(-,e_{l-1})| \id | \ldots | \id$ is a homotopy-invertible chain map, so the summand is contractible. (We used a similar argument in the proof of Lemma~\ref{lemma. right inverse}.)

And of course, $F_{\leq 1}$ is precisely $\hom_A(X,Y)$. We conclude that each of the maps
	\eqnn
	\hom_A(X,Y) 
	\to
	\hom_{A[\id_A^{-1}]}(X,Y)
	\to 
	\hom_{\Aploc}(X,Y)/\II(X,Y) 
	\eqnnd
are quasi-isomorphisms. By applying Lemma~\ref{lemma. right inverse}, the result follows.
\end{proof}

\begin{remark}
Though the natural inclusion $A \to \Aploc$ is not a strictly unital functor, if we compose with the quotient map, the result
	\eqn\label{eqn. A to mod I}
	A \to \Aploc/\II
	\eqnd
is strictly unital. We see~\eqref{eqn. A to mod I} is natural in the $A$ variable by combining~\eqref{eqn. naturality of tau} and~\eqref{eqn. naturality of mod I}.
\end{remark}

\begin{proof}[Alternate proof of Lemma~\ref{lemma. j has left inverse}.]
We have natural transformations of functors (in the classical sense) 
	\eqn
	\label{eqn. natural transformations for left inverse}
	\id_{\Ainftystr} \xrightarrow{\eqref{eqn. A to mod I}} \tau_{/\II} \xleftarrow{\text{Remark~\ref{remark. nat trans tau to tau mod I}}} \tau \circ j.
	\eqnd
The naturality of \eqref{eqn. A to mod I} -- along with the two-out-of-three property for quasi-equivalences -- shows that $\tau_{/\II}$ respects quasi-equivalences, so induces an endofunctor of $\Ainftystr[\eqs^{-1}]$. By Lemma~\ref{lemma. mod I is quasi equivalence}  and Proposition~\ref{prop. good natural transformations induce homotopies}, the natural transformations induced by~\eqref{eqn. natural transformations for left inverse} are natural equivalences upon localizing along $\eqs$. 

Thus, the functor induced by $\tau \circ j$ is naturally equivalent to the identity functor, and this exhibits the functor induced by $\tau$ as a left inverse to the functor induced by $j$.
\end{proof}

\bibliographystyle{plain}
\bibliography{20240705-biblio}

\end{document}